\documentclass{amsart}

\usepackage{amsmath}
\usepackage{amsthm}
\usepackage{amsopn}
\usepackage{amssymb}
\usepackage[all]{xy}
\usepackage{graphicx}
\usepackage{mathtools}

\allowdisplaybreaks[1]

\def \mmod{/\mkern-3mu /}

\newcommand{\mc}[1]{\mathcal{#1}}

\newcommand{\mr}[1]{\mathrm{#1}}

\newcommand{\mit}[1]{\mathit{#1}}

\newcommand{\abs}[1]{\lvert #1 \rvert}

\newcommand{\bra}[1]{\langle #1 \rangle}
\newcommand{\br}[1]{\overline{#1}}
\newcommand{\ul}[1]{\underline{#1}}

\newcommand{\td}[1]{\widetilde{#1}}

\newcommand{\ZZ}{\mathbb{Z}}

\newcommand{\CC}{\mathbb{C}}
\newcommand{\QQ}{\mathbb{Q}}

\newcommand{\FF}{\mathbb{F}}

\newcommand{\TMF}{\mathrm{TMF}}
\newcommand{\Tmf}{\mathrm{Tmf}}
\newcommand{\tmf}{\mathrm{tmf}}
\newcommand{\bo}{\mathrm{bo}}
\newcommand{\bou}{\underline{\bo}}
\newcommand{\xib}{\bar{\xi}}

\theoremstyle{definition}

 \newtheorem{thm}[equation]{Theorem}
 \newtheorem{cor}[equation]{Corollary}
 \newtheorem{lem}[equation]{Lemma}
 \newtheorem{prop}[equation]{Proposition}
 
 \newtheorem{ex}[equation]{Example}
 
 \newtheorem{rmk}[equation]{Remark}
\newtheorem*{thm*}{Theorem}
\newtheorem*{cor*}{Corollary}
\newtheorem*{lem*}{Lemma}
\newtheorem*{prop*}{Proposition}
\newtheorem*{defn*}{Definition}
\newtheorem*{ex*}{Example}
\newtheorem*{exs*}{Examples}
\newtheorem*{rmk*}{Remark}
\newtheorem*{claim*}{Claim}

\numberwithin{equation}{section}
\numberwithin{figure}{section}
\DeclareMathOperator{\Ext}{Ext}
\DeclareMathOperator{\coker}{coker}

\DeclareMathOperator{\im}{im}

\DeclareMathOperator*{\hocolim}{hocolim}
\DeclareMathOperator*{\colim}{colim}

\title{Detecting exotic spheres in low dimensions using coker J}

\author[M.~Behrens]{M.~Behrens$\sp 1$}
\address{
Dept. of Mathematics \\
University of Notre Dame \\
Notre Dame, IN, U.S.A.
}
\author[M.~Hill]{M.~Hill$\sp 2$}
\address{
Dept. of Mathematics \\
UCLA \\
Los Angeles, CA, U.S.A.
}
\author[M.J.~Hopkins]{M.J.~Hopkins$\sp 3$}
\address{
Dept. of Mathematics \\
Harvard University \\
Cambridge, MA, U.S.A.
}
\author{M.~Mahowald}
\address{
Dept. of Mathematics \\
Northwestern University \\
Evanston, IL, U.S.A.
}

\begin{document}

\begin{abstract}
Building off of the work of Kervaire-Milnor and Hill-Hopkins-Ravenel, Wang and Xu showed that the only odd dimensions $n$ for which $S^n$ has a unique differentiable structure are 1, 3, 5, and 61.  We show that the only even dimensions below 140 for which $S^n$ has a unique differentiable structure are 2, 6, 12, 56, and perhaps 4.
\end{abstract}

\footnotetext[1]{The author was supported by the NSF under grants DMS-1050466, DMS-1452111, and DMS-1611786.}
\footnotetext[2]{The author was supported by the NSF under grants DMS-1207774, DMS-1509652, and DMS-1811189.}
\footnotetext[3]{The author was supported by the NSF under grants DMS-1510417 and DMS-1810917.}

\maketitle

\tableofcontents

\section{Introduction}

A homotopy $n$-sphere is a smooth $n$-manifold which is homotopy equivalent to $S^n$.
Kervaire and Milnor defined $\Theta_n$ to be the group of homotopy spheres up to $h$-cobordism (where the group operation is given by connect sum).  By the $h$-cobordism theorem \cite{Smale} ($n > 4$) and Perelman's proof of the Poincare conjecture \cite{Perelman1}, \cite{Perelman3}, \cite{Perelman2} ($n = 3$), for $n \ne 4$, $\Theta_n = 0$ if and only if $S^n$ has a unique differentiable structure (i.e. there are no exotic spheres of dimension $n$).  

We wish to consider the following question:
$$
\text{\it For which $n$ is $\Theta_n = 0$?}
$$
The general belief is that there should be finitely many such $n$, and these $n$ should be concentrated in relatively low dimensions.

Kervaire and Milnor \cite{KervaireMilnor}, \cite{Milnor} showed that there is an isomorphism 
$$ \Theta_{4k} \cong \coker J_{4k} $$
and there are exact sequences
\begin{gather*}
0 \rightarrow \Theta_{2k+1}^{bp} \rightarrow \Theta_{2k+1} \rightarrow \coker J_{2k+1} \rightarrow 0, \\
0 \rightarrow \Theta_{4k+2} \rightarrow \coker J_{4k+2} \xrightarrow{\Phi_K} \ZZ/2 \rightarrow \Theta^{bp}_{4k+1} \rightarrow 0. \\
\end{gather*}
Here $\Theta_n^{bp}$ is the subgroup of those homotopy spheres which bound a parallelizable manifold, $\coker J$ is the cokernel of the $J$ homomorphism
$$ J: \pi_{n}SO \rightarrow \pi_n^s, $$
(where $\pi_*^s$ denotes the stable homotopy groups of spheres, a.k.a. the stable stems) and $\Phi_K$ is the Kervaire invariant.

Kervaire and Milnor showed that $\Theta_{4k+3}^{bp}$ is non-trivial for $k \ge 1$.  The non-triviality of $\Theta_{4k+1}^{bp}$ depends on the non-triviality of the Kervaire homomorphism $\Phi_K$.  The following theorem of Browder is essential \cite{Browder}.

\begin{thm}[Browder]\label{thm:Browder}
$x \in \pi_*^s$ has $\Phi_K(x) = 1$ if and only if it is detected in the Adams spectral sequence by the class $h_j^2$.  In particular, Kervaire invariant one elements can only occur in dimensions of the form $2^i-2$.
\end{thm}

The work of Hill-Hopkins-Ravenel \cite{HHR} implies that $\Phi_K$ is non-trivial only in dimensions 2, 6, 14, 30, 62, and perhaps 126.  This implies that $\Theta^{bp}_{4k+1}$ is non-trivial, except in dimensions 1, 5, 13, 29, 61, and perhaps 125.  In dimensions 13 and 29, well established computations of $\pi_*^s$ show $\coker J$ is non-trivial.  With extraordinary effort, Wang and Xu recently showed $\pi^s_{61} = 0$ \cite{WangXu}.  They also observe that $\coker J_{125} \ne 0$ by producing an explicit element whose non-triviality is detected by $\tmf$, the spectrum of topological modular forms.  This concludes the analysis of odd dimensions.

The case of even dimensions, by the Kervaire-Milnor exact sequence, boils down to the question: 
$$
\text{for which $k$ does $\coker J_{2k}$ have a non-trivial element of Kervaire invariant $0$?}
$$
\emph{The purpose of this paper is to examine the extent to which modern techniques in homotopy theory can address this question.}

Over the years we have amassed a fairly detailed knowledge of the stable stems in low degrees using the Adams and Adams-Novikov spectral sequences.  We refer the reader to \cite{Ravenel} for a good summary of the state of knowledge at odd primes, and \cite{Isaksen}, \cite{Isaksenchart} for a detailed account of the current state of affairs at the prime $2$.  In particular, we have a complete understanding of $\pi_n^s$ in a range extending somewhat beyond $n = 60$.\footnote{Using a recent breakthrough of Gheorghe-Wang-Xu \cite{GheorgheWangXu}, Isaksen-Wang-Xu are now using motivic homotopy theory and machine computations to bring the range beyond $n = 90$.}

However, we do not need to completely compute $\pi_n^s$ to simply deduce $\coker J_n$ has a non-zero element of Kervaire invariant $0$ --- it suffices to produce a single non-trivial element.  For example, Adams \cite[Theorem~1.2]{Adams} produced families of elements in $\coker J_{n}$, for $n \equiv 1,2 \pmod 8$ (see Section~\ref{sec:J}) whose non-triviality is established by observing that they have non-trivial image under the Hurewicz homomorphism for real K-theory ($KO$) 
$$ \pi_n^s \rightarrow \pi_*KO. $$

Adams' families constitute examples of \emph{$v_1$-periodic families}.  
Quillen \cite{Quillen} showed that the $E_2$-term of the Adams-Novikov spectral sequence
$$ \Ext_{MU_*MU}(MU_*, MU_*) \Rightarrow \pi_*^s $$
(the Adams spectral sequence based on complex cobordism $MU$) can be described in terms of the moduli space of formal group laws. 
Motivated by ideas of Morava, Miller-Ravenel-Wilson \cite{MRW} lifted the stratification of the $p$-local moduli of formal groups by height to show that for a prime $p$, the $p$-localization of this $E_2$-term admits a filtration (called the \emph{chromatic filtration}) where the $n$th layer consists of periodic families of elements (these are called $v_n$-periodic families).
By proving a series of conjectures of Ravenel \cite{Ravenel84}, Devinatz, Hopkins, and Smith \cite{Nilp1}, \cite{Nilp2} showed that this chromatic filtration lifts to a filtration of the stable stems $(\pi_*^s)_{(p)}$. Thus the stable stems decompose into chromatic layers, and the $n$th layer is generated by $v_n$-periodic families of elements.

Unfortunately (see \cite[Ch.~5.3]{Ravenel}), the Adams families described above constitute the only $v_1$-periodic elements which are not in the image of the $J$-homomorphism (and every element in the image of the $J$ homomorphism is $v_1$-periodic).
Therefore, $\coker J$ is generated by the Adams families, and the $v_n$-periodic families with $n \ge 2$.

In this paper we will produce such non-trivial elements of $\coker J$ in low degrees using two techniques:
\begin{enumerate}
\item Take a product or Toda bracket of known elements in $\pi_*^s$ to produce a new element in $\pi_*^s$, and show the resulting element is non-trivial in $\coker J$.

\item Use chromatic homotopy theory to produce non-trivial $v_2$-periodic families.
\end{enumerate}
Just as the Adams families are detected by the $KO$ Hurewicz homomorphism, many $v_2$-periodic families are detected by the theory of topological modular forms (tmf).

The main theorem is the following.

\begin{thm}\label{thm:main}
For every even $k < 140$, $\coker J_{k}$ has a non-trivial element of Kervaire invariant $0$, except for $k = $ 2, 4, 6, 12, and 56.
\end{thm}

Combining this with the discussion at the beginning of this section yields the following corollary.

\begin{cor}
The only dimensions less than 140 for which $S^n$ has a unique differentiable structure are 1,2,3,5, 6, 12, 56, 61, and perhaps 4.
\end{cor}

Theorem~\ref{thm:main} will be established in part by using the complete computation of $(\pi_*^s)_{(2)}$ for $* < 60$, $(\pi_*^s)_{(3)}$ for $* < 104$, and a small part of the vast knowledge of $(\pi_*^s)_{(5)}$ (computed in \cite{Ravenel} for $* < 1000$), though the $5$ torsion is very sparse, and contributes very little to the discussion.  Contributions from primes greater than $5$ offer nothing to this range ($(\coker J_{82})_{(7)}$ is non-trivial, but we handle this dimension through other means). 

Additional work will be required to produce non-trivial classes in $(\coker J_*)_{(2)}$ for $* \ge 60$.  
Regarding technique (1) above, when taking a product of known elements, we must establish the resulting product is non-trivial.  This will be done in some cases by observing these products are detected by classes in the Adams spectral sequence which cannot be targets of differentials.  In other cases, this will be done by observing that the image of the resulting product under the $\tmf$ Hurewicz homomorphism is non-zero (in one case we instead need to use the fiber of a certain map between $\tmf$-spectra).

However, the bulk of the work in this paper will be devoted to pursuing technique (2) above.  Specifically we will lift non-trivial elements in $\pi_*\tmf$ to $v_2$-periodic elements in $\pi_*^s$.

Most of the classes we construct are $v_2$-periodic - these periodic classes actually imply the existence of exotic spheres in infinitely many dimensions, and limit the remaining dimensions to certain congruence classes, though we do not pursue this here.

For some time we did not know how to construct a non-trivial class in $\coker J$ of Kervaire invariant $0$ in dimension $126$, but Dan Isaksen and Zhouli Xu came up with a clever argument which handles that case (Theorem~\ref{thm:126}). 

We do not know if there are any non-trivial classes in $\coker J_{140}$. This is why we stop there.  But we do end the paper with some remarks which explain that there are actually only a handful of dimensions in the range $140-200$ where we are unable to produce non-trivial classes in $\coker J$. 
Some of these classes were communicated to the authors by Zhouli Xu.
\vspace{24pt}

\subsection*{Organization of the paper}$\quad$

In Section~\ref{sec:piS}, we recall some computations of the $2$-component of the stable homotopy groups of spheres in low dimensions.

In Section~\ref{sec:vn}, we recall the notion of $v_n$-periodicity in the stable stems.

In Section~\ref{sec:J}, we recall facts about $v_1$-periodicity and its relationship to the image of the J homomorphism, and discuss how the Adams families give non-trivial elements of $\coker J$ in degrees congruent to $2$ mod $8$.

In Section~\ref{sec:tmf}, we recall some facts about the theory of topological modular forms (tmf).

In Section~\ref{sec:strategy}, we discuss the method we will use to produce $v_2$-periodic elements by lifting them from the homotopy groups of $\tmf$.
This method will involve producing elements in the homotopy groups of a certain type 2 complex $M(8,v_1^8)$. These elements map to the desired elements after projecting onto the top cell of the complex.   These elements will be produced using a \emph{modified Adams spectral sequence} (MASS).  The $E_2$-term of the MASS will be analyzed by means of the \emph{algebraic tmf resolution}.

In Section~\ref{sec:M38}, we establish many important properties of $M(8,v_1^8)$ which we will need.

In Section~\ref{sec:boi} we discuss the computation of the $E_1$-page of the algebraic tmf resolution by means of computing Ext of \emph{bo-Brown-Gitler modules}.

In Section~\ref{sec:tmfM38}, we compute the modified Adams spectral sequence for $\tmf_*M(8,v_1^8)$.  This allows us to identify classes in $\tmf_*M(8,v_1^8)$ which map to the desired classes in $\tmf_*$ after projection to the top cell.

In Section~\ref{sec:tmfres}, we compute enough of the algebraic tmf-resolution to show that the desired elements persist in the algebraic tmf-resolution.

In Section~\ref{sec:MASS}, we show the desired elements are permanent cycles in the MASS.

In Section~\ref{sec:cokerJ} we tabulate the non-trivial elements of Coker J comprised of the elements produced in the previous sections, together with some other elements produced by ad hoc means, in dimensions less than $140$.
We also include a tentative discussion of the state of affairs below dimension 200.

\subsection*{Conventions}

Throughout this paper we will let $\{x\} \subset \pi_*X$ denote a coset of elements detected by an element $x$ in the $E_2$-term of an Adams spectral sequence (ASS) or Adams-Novikov spectral sequence (ANSS).  
Conversely, for a class $\alpha$ in homotopy, we let $\ul{\alpha}$ denote an element in the ASS which detects it. We let $A_*$ denote the dual Steenrod algebra, $A\mmod A(2)_*$ denote the dual of the Hopf algebra quotient $A\mmod A(2)$, and for an $A_*$-comodule $M$ (or more generally an object of the derived category of $A_*$-comodules) we let
$$ \Ext^{s,t}_{A_*}(M) $$
denote the group $\Ext^{s,t}_{A_*}(\FF_2, M)$.  For a spectrum $E$, we let $E_*$ denote its homotopy groups $\pi_*E$.

\subsection*{Acknowledgments} The authors would like to thank John Milnor, who suggested this project to the third author.  Dan Isaksen and Zhouli Xu provided valuable input, and in particular produced the argument which resolves dimension $126$.  The authors also benefited from comments and corrections from Achim Krause and Larry Taylor.  Finally, this paper was greatly improved by thoughtful suggestions of three referees.

\section{Low dimensional computations of $(\pi_*^s)_{(2)}$}\label{sec:piS}
 
\begin{figure}
\includegraphics[angle = 90, origin=c, height =.8\textheight]{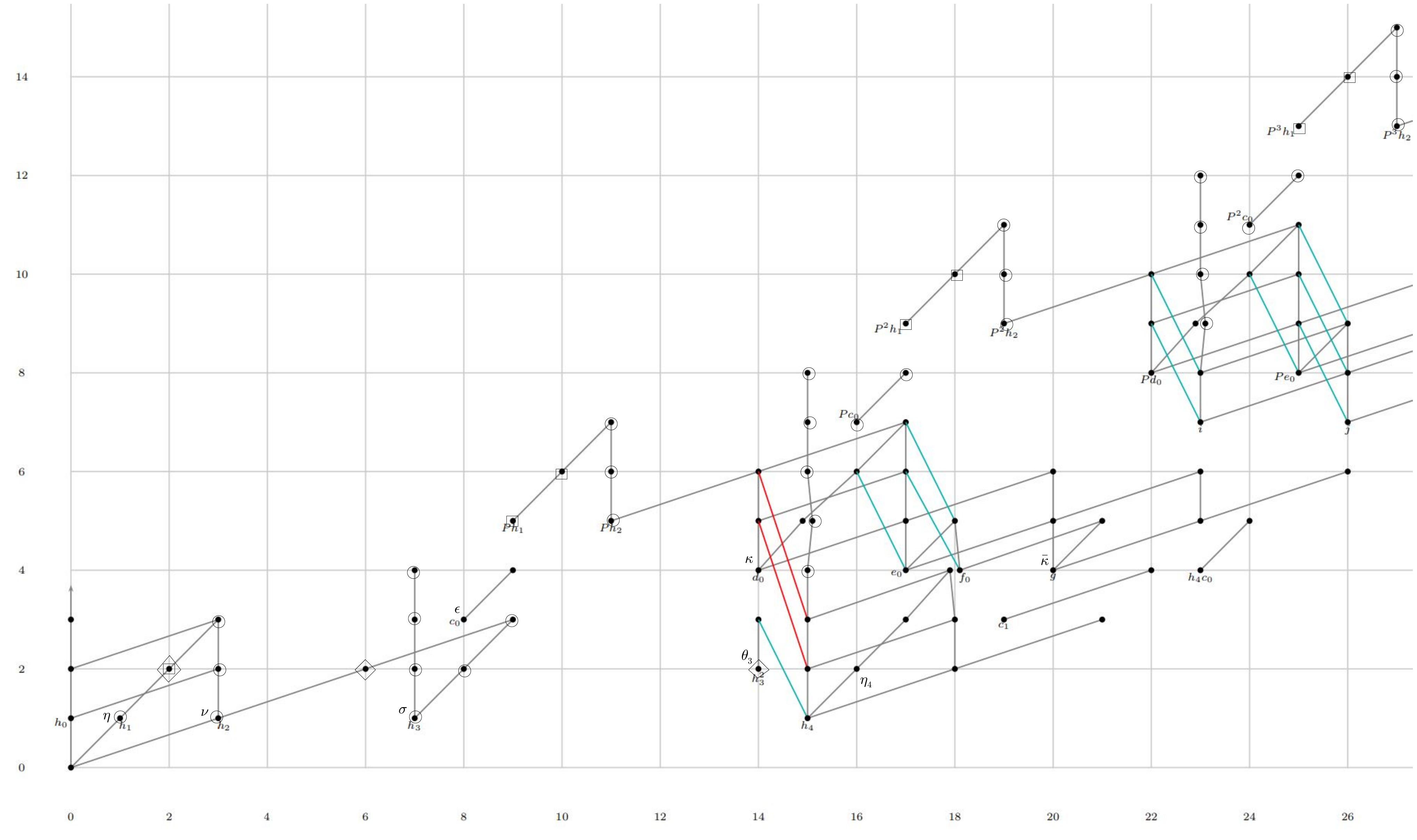}
\caption{The $2$-primary ASS for the sphere in the region $0 \le t-s \le 26$ (courtesy of Dan Isaksen).}\label{fig:ASS26}
\end{figure} 
 
\begin{figure}
\includegraphics[angle = 0, origin=c, height =.9\textheight]{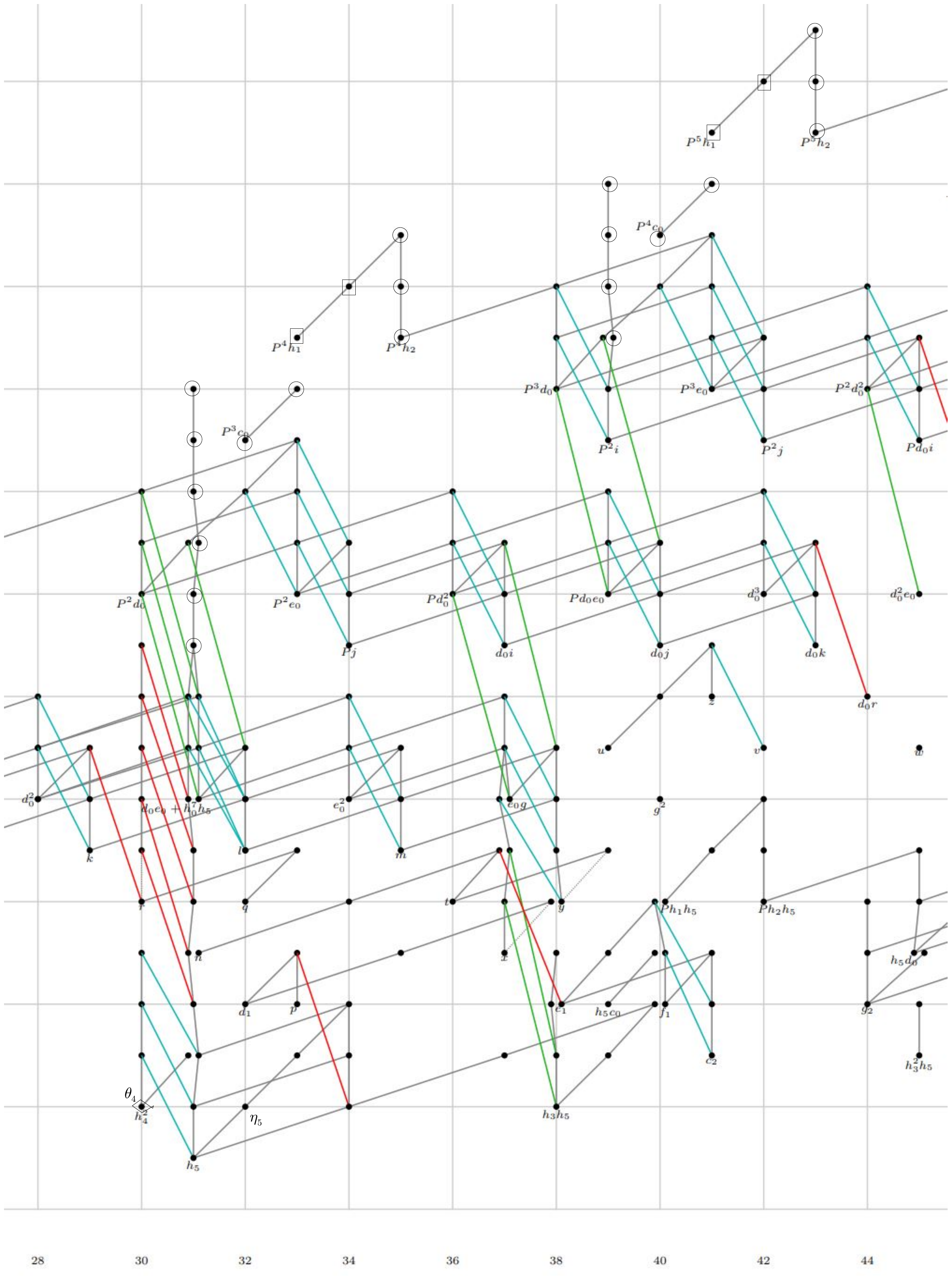}
\caption{The $2$-primary ASS for the sphere in the region $26 \le t-s \le 44$ (courtesy of Dan Isaksen).}\label{fig:ASS44}
\end{figure}  

For the convenience of the reader, we summarize some low dimensional computations of $(\pi_*^s)_{(2)}$.  Generators in this range will play an important part in the remainder of this paper.

Figures~\ref{fig:ASS26} and \ref{fig:ASS44} depict the mod $2$ Adams spectral sequence (ASS)
$$ E_2^{s,t} = \Ext^{s,t}_{A_*}(\FF_2, \FF_2) \Rightarrow (\pi_{t-s}^{s})^{\wedge}_2 $$
in the range $t-s \le 44$.  In these charts, generators are depicted in bidegrees $(t-s,s)$.  Dots represent factors of $\FF_2$.  Lines of non-negative slope represent multiplication by $h_0$, $h_1$, and $h_2$ in the $E_2$-term.  Lines of negative slope depict differentials
$$ d_r: E_r^{s,t} \rightarrow E_r^{s+r,t+r-1}. $$ 

The element
$$ h_0 \in \Ext^{1,1}_{A_*}(\FF_2, \FF_2) $$
detects the degree $2$ map in $\pi_0^s$.  Therefore, if a permanent cycle $x \in E_2^{s,t}$ detects $\alpha \in \pi^s_{t-s}$, then if the element $h_0 x$ is non-trivial in $E^{s+1,t+1}_\infty$, it detects $2\alpha \in \pi^s_{t-s}$.  For example, we see from Figure~\ref{fig:ASS26} that $(\pi^s_3)^{\wedge}_2 \cong \ZZ/8$.  

The elements $h_i \in \Ext^{1,2^i}_{A_*}$ detect elements of $\pi^s_*$ of Hopf invariant $1$.  Adams' work on the Hopf invariant one problem \cite{AdamsHI} implies that the elements $h_i$ support differentials for $i \ge 4$.

In Figures~\ref{fig:ASS26} and \ref{fig:ASS44}, we use circles to denote elements which detect elements in the image of the $J$-homomorphism
$$ \pi_*SO \rightarrow \pi_*^s. $$
These elements fit into a predictable pattern which continues with slope $1/2$ in the Adams spectral sequence chart \cite{MahowaldJ}, \cite{DavisMahowald}.  This pattern an instance of $v_1$-periodicity, which will be explained further in Section~\ref{sec:vn}.  The elements decorated with boxes are $v_1$-periodic, but are not in the image of the $J$-homomorphism.  By Browder's theorem (Theorem~\ref{thm:Browder}), an element in $\pi_*^s$ has Kervaire invariant one if and only if it is detected by an elements $h_j^2$ in the Adams spectral sequence --- these classes are denoted with diamonds.  Thus only dots which are neither marked with a circle or a diamond contribute to the group $\Theta_n$.   

Some generators of $\pi_*^s$ are famous, and are known by names ascribed to them by Toda \cite{Toda}.  We will often refer to such generators by these names.  The table below gives a dictionary with associates a name to an element in the Adams spectral sequence which detects it, and the corresponding dots in Figure~\ref{fig:ASS26} are labeled both with the traditional Ext name, and a Toda name of an element it detects.

\begin{center}
\begin{tabular}{c|c|c}
degree & Toda name & ASS generator \\
\hline
1 & $\eta$ & $h_1$ \\
3 & $\nu$ & $h_2$ \\
7 & $\sigma$ & $h_3$ \\
8 & $\epsilon$ & $c_0$ \\
14 & $\kappa$ & $d_0$ \\
20 & $\bar{\kappa}$ & $g$ \\
$2^{j+1}-2$ & $\theta_j$ & $h_j^2$ \\
$2^{j}$ & $\eta_j$ & $h_1h_j$ \\
\end{tabular}
\end{center}

\section{$v_n$-periodicity}\label{sec:vn}

We recall that there are two formulations of the chromatic filtration in the stable stems: one coming from the chromatic tower, and the second coming from the Hopkins-Smith Periodicity Theorem (see \cite{Ravenelorange} for a more detailed exposition of this discussion).

\subsection*{The chromatic tower}

Fix a prime $p$. The chromatic tower of a spectrum $X$ is the tower of Bousfield localizations
$$ X \rightarrow \cdots \rightarrow X_{E(n)} \rightarrow X_{E(n-1)} \rightarrow \cdots \rightarrow X_{E(0)} $$
where $E(n)$ is the $n$th Johnson-Wilson  spectrum ($E(0) = H\QQ$, by convention) with
$$ E(n)_* = \ZZ_{(p)}[v_1, \ldots, v_{n-1}, v_n^{\pm}]. $$
The fibers of the chromatic tower
$$ M_n X \rightarrow X_{E(n)} \rightarrow X_{E(n-1)} $$
are called the \emph{monochromatic layers}.  The spectral sequence associated to the chromatic tower is the chromatic spectral sequence
$$ E_1^{n,*} = \pi_*M_nX \Rightarrow \pi_*X_{(p)}. $$
The Hopkins-Ravenel Chromatic Convergence Theorem \cite{HopkinsRavenel} states that this spectral sequence converges for $X$ finite.  We shall say an element $x$ in $\pi_*X_{(p)}$ has \emph{chromatic filtration $n$} if it is detected on the $n$-line $\pi_*M_nS$ of the chromatic spectral sequence.

The efficacy of the chromatic spectral sequence (in the case of $X = S$) is that the $E(n)$-based Adams-Novikov spectral sequences
$$ \Ext_{E(n)_*E(n)}(E(n)_*, E(n)_*/(p^\infty, v_1^{\infty}, \ldots, v_{n-1}^{\infty})) \Rightarrow \pi_*M_nS $$
converge at finite stages and are in principle completely computable (though at present the vast majority of such computations have been carried out for $n \le 2$).
For given $i_0, \ldots i_{n-1}$, there exists a $k$ so that the element 
$$ v^{k}_n \in E(n)_*/(p^{i_0}, \ldots, v_{n-1}^{i_{n-1}}) $$ 
is primitive with respect to the $E(n)_*E(n)$-coaction.  Since $v_n$ is invertible, it follows that the Ext groups
$$ \Ext_{E(n)_*E(n)}(E(n)_*, E(n)_*/(p^{i_0}, v_1^{i_1}, \ldots, v_{n -1}^{i_{n-1}})) $$
are $v_n^k$-periodic.  Since
\begin{multline*}
\Ext_{E(n)_*E(n)}(E(n)_*, E(n)_*/(p^\infty, v_1^{\infty}, \ldots, v_{n-1}^{\infty})) = \\
\colim_{{i_0, i_1, \cdots, i_{n-1}}} \Ext_{E(n)_*E(n)}(E(n)_*,  E(n)_*/(p^{i_0}, v_1^{i_1}, \ldots, v_{n-1}^{i_{n-1}})),
\end{multline*}
this suggests that elements of chromatic filtration $n$ fit into periodic families.

\subsection*{The telescopic chromatic tower}

Bousfield localization $X_E$ is obtained by killing $E$-acyclic spectra.  There is a variant $X^f_E$, called \emph{finite localization}, where one kills only finite $E$-acyclic spectra \cite{Millerfinite}.  The telescope conjecture hypothesizes that in the case of $E(n)$, the map
$$ X^f_{E(n)} \rightarrow X_{E(n)} $$
is an equivalence.  The telescope conjecture has been proven in the case of $n = 1$ \cite{Mahowaldbo} \cite{Miller}, but is believed to be false for $n \ge 2$ \cite{MahowaldRavenelShick}.

Associated to this different type of localization is the \emph{telescopic chromatic tower}
$$ X \rightarrow \cdots \rightarrow X^f_{E(n)} \rightarrow X^f_{E(n-1)} \rightarrow \cdots \rightarrow X^f_{E(0)}. $$
We shall denote the fibers by $M_n^fX$.  

The relevance of this telescopic variant is that the monochromatic layers are directly constructed from topological analogs of the algebraic constructs of the previous subsection.  Specifically, the Hopkins-Smith periodicity theorem \cite{Nilp2} implies, for a cofinal sequence of multi-indices $(i_0, \ldots, i_{n-1})$, the existence of \emph{generalized Moore spectra}
$$ M(p^{i_0}, v_1^{i_1}, \ldots, v_{n-1}^{i_{n-1}}) $$
which have the property that 
$$ E(n)_*M(p^{i_0}, v_1^{i_1}, \ldots, v_{n-1}^{i_{n-1}}) \cong E(n)_*/(p^{i_0}, v_1^{i_1}, \ldots, v_{n-1}^{i_{n-1}}). $$
These spectra are examples of \emph{type $n$ spectra}.  The periodicity theorem states that all such spectra admit $v_n$-self maps: specialized to our situation,  
for certain $k$ these spectra have \emph{$v_n$-self maps} 
$$ v_n^k: \Sigma^{2k(p^n-1)}M(p^{i_0}, v_1^{i_1}, \ldots, v_{n-1}^{i_{n-1}}) \rightarrow M(p^{i_0}, v_1^{i_1}, \ldots, v_{n-1}^{i_{n-1}}) $$
which induce multiplication by $v_n^k$ on $E(n)$-homology.  
There is an equivalence 
\begin{equation}\label{eq:Mnf}
 M^f_nX \simeq \hocolim_{(i_0, \ldots, i_{n-1})} M^0(p^{i_0}, v_1^{i_1}, \ldots, v_{n-1}^{i_{n-1}})^f_{E(n)} \wedge X
 \end{equation}
(where the superscript zero above indicates that the finite spectrum is desuspended so that its top cell lies in dimension $0$).
The term ``telescopic'' comes from the formula \cite{Millerfinite}
\begin{equation}\label{eq:telescope}
M(p^{i_0}, v_1^{i_1}, \ldots, v_{n-1}^{i_{n-1}})^f_{E(n)} \simeq (v_n^k)^{-1}M(p^{i_0}, v_1^{i_1}, \ldots, v_{n-1}^{i_{n-1}})
\end{equation}
where the localization on the right hand side is obtained by taking the homotopy colimit (telescope) of iterates of the $v_n$-self map. 
It follows that the homotopy groups of the finite localization above are $v_n^k$-periodic.  We shall therefore say that an element $x$ in $\pi_*X_{(p)}$ is \emph{$v_n$-periodic} if it is detected in $\pi_*M^f_nX$ in the telescopic chromatic tower. 

This $v_n^k$-periodicity is reflected in the (non-telescopic) monochromatic layers, since the $v_n$-self maps are $E(n)$-equivalences, giving
$$ M(p^{i_0}, v_1^{i_1}, \ldots, v_{n-1}^{i_{n-1}})_{E(n)} \simeq (v_n^k)^{-1}M(p^{i_0}, v_1^{i_1}, \ldots, v_{n-1}^{i_{n-1}})_{E(n)}. $$
Note that $E(n)$-localization is smashing \cite{Nilp2}, and therefore it commutes with homotopy colimits.  We also therefore have 
$$ M_nX \simeq \hocolim_{(i_0, \ldots, i_{n-1})} M^0(p^{i_0}, v_1^{i_1}, \ldots, v_{n-1}^{i_{n-1}})_{E(n)} \wedge X. $$

\subsection*{$v_n$-periodic families from finite complexes}

We will define
$$ M^0(p^\infty, \ldots, v_n^\infty) := \hocolim_{(i_0, \ldots, i_{n})} 
M^0(p^{i_0}, \ldots, v_{n}^{i_{n}}). $$
By taking colimits over $(i_0, \ldots, i_n)$ of the fiber sequences 
$$ M^0(p^{i_0}, \ldots, v_n^{i_n}) \rightarrow M^0(p^{i_0}, \ldots, v_{n-1}^{i_{n-1}}) \xrightarrow{v_n^{i_n}} \Sigma^{-i_n\abs{v_n}} M^0(p^{i_0}, \ldots, v_{n-1}^{i_{n-1}})
$$
and applying (\ref{eq:Mnf}) and (\ref{eq:telescope}), we get fiber sequences
$$ M^0(p^\infty, \ldots, v_n^\infty) \rightarrow M^0(p^\infty, \ldots, v_{n-1}^\infty) \rightarrow M_n^fS. $$
It follows that there are fiber sequences
$$ M^0(p^\infty, \ldots, v_n^\infty) \wedge X \xrightarrow{t} X \rightarrow X^f_{E(n)} $$
(where $t$ is projection on the top cell of the first term).  
Suppose that $x \in \pi_*S_{(p)}$ is $v_n$-periodic.  Then
$x$ lifts to an element
$$ \td{x} \in \pi_*M^0(p^{i_0}, \ldots, v^{i_{n-1}}_{n-1}) \wedge X $$
with $v_n^{sk}\td{x} \ne 0$ for all $s$ (where the generalized Moore complex above has a $v_n^k$-self map).
Let $v_n^{sk} x \in \pi_* X$ denote the image of $v_n^{sk}\td{x}$ under the projection to the top cell of the generalized Moore complex (note that this notation hides the fact that many choices were made, such as the homotopy type of the generalized Moore spectrum, the lift $\td{x}$, and the $v_n^k$-self map).
We shall call the family $\{v_n^{sk} x\}_{s \ge 0}$ a \emph{$v_n$-periodic family generated by $x$}.  Note that it may be that for $s>0$, some or all of the elements $v_n^{sk}x$ might actually be null.  However, we do at least have the following.

\begin{lem}
If the element $v_n^{sk}x \in \pi_*X$ defined above is not null, then it is $v_n$-periodic.
\end{lem}

\begin{proof}
The assumption that $x$ is $v_n$-periodic implies that the image of $x$ in $M_n^fX$ is non-trivial.  Then it follows that the image of $v_n^{sk}x$ in $M_n^fX$ is non-trivial.
\end{proof}

One notable potential difference between the telescopic chromatic tower and the chromatic tower is that $v_n$-periodic elements generate $v_n$-periodic families, whereas elements of chromatic filtration $n$ simply generate $v_n$-periodic families in $X_{E(n)}$ which do not necessarily lift to $\pi_*X$.  Note that since there is a map
$$ X^f_{E(n)} \rightarrow X_{E(n)}, $$
a $v_n$-periodic element has chromatic filtration greater than or equal to $n$. 

One explicit way of producing $v_n$-periodic elements is simply to produce an element 
$$ y \in \pi_*M(p^{i_0}, \ldots, v_{n-1}^{i_{n-1}}) \wedge X $$
for which the projections of the $v_n$-iterates $v_n^{sk} y$ to the top cell of the generalized Moore spectrum are all non-zero.  For example, the $p$-component of the image of the $J$ homomorphism ($p$ odd) is generated by the various $v_1$-periodic composites
$$ \alpha_{k/i} :S^{2k(p-1)} \hookrightarrow \Sigma^{2k(p-1)} M(p^i) \xrightarrow{v_1^{k}} M(p^i) \rightarrow S^1 $$
(where $k$ is chosen so that the $v_1$-self map exists).

This construction has an obvious generalization --- we may define $v_2$-periodic elements $\beta_{k/j,i} \in \pi_*^s$ to be the composites
$$ \beta_{k/i,j}: S^{k\abs{v_2}} \hookrightarrow \Sigma^{k\abs{v_2}} M(p^i,v_1^j) \xrightarrow{v_2^k} M(p^i, v_1^j) \rightarrow S^{j\abs{v_1}+2} $$
These were shown in \cite{MRW} to be non-trivial for all combinations $(i,j,k)$ for which these elements exist.  One denotes $\beta_{k/j,1}$ by $\beta_{k/j}$, and $\beta_{k/1}$ is simply denoted $\beta_k$.  These elements will appear in the tables of Section~\ref{sec:cokerJ}.

In general, the composites 
\begin{multline*} \alpha^{(n)}_{k/i_{n-1}, \ldots i_{0}}: S^{k\abs{v_n}} \hookrightarrow M(p^{i_0}, \ldots, v_{n-1}^{i_{n-1}}) \xrightarrow{k\abs{v_n}} M(p^{i_0}, \ldots, v^{i_{n-1}}_{n-1}) \\ \rightarrow S^{i_1\abs{v_1}+\cdots+i_{n-1}\abs{v_{n-1}}+n}
\end{multline*}
are called the \emph{$n$th Greek letter elements} (here $\alpha^{(n)}$ is the $n$th letter in the Greek alphabet).  We know very little about these $v_n$-periodic families for $n \ge 3$. 

\section{$v_1$-periodicity and the image of $J$}\label{sec:J}

Fix a prime $p$, and let $\ell$ be a topological generator of $\ZZ_p^\times/\{\pm 1\}$.  The $(p)$-local $J$ spectrum is defined to be the fiber
\begin{equation}\label{eq:J}
 J_{(p)} \rightarrow KO_{(p)} \xrightarrow{\psi^\ell - 1} KO_{(p)}
 \end{equation}
where $\psi^\ell$ is the $\ell$th stable Adams operation.

Using work of Adams-Baird and Ravenel, Bousfield \cite{Bousfield} shows there is a fiber sequence
$$ S_{E(1)} \rightarrow J_{(p)} \rightarrow \Sigma^{-1}H\QQ $$
so that the following diagram commutes.
$$
\xymatrix{
\pi_*SO \ar[r] \ar[d]_{J} & 
\pi_*\Sigma^{-1}KO_{(p)} \ar[dd]^{\partial}
\\
\pi_*S  \ar[d] 
\\
\pi_*S_{E(1)} \ar[r] &
\pi_*J_{(p)}
}
$$
Here, $\partial$ is the boundary homomorphism associated to the fiber sequence (\ref{eq:J}).  The work of Adams \cite{Adams} shows that the resulting map
$$ (\im J)_{k} \rightarrow \pi_k S_{E(1)} $$
is an isomorphism for $p$ odd and $k > 0$ (and an injection for $p = 2$).
As explained in the introduction, we are only concerned with $\coker J$ in even dimensions.  The only non-trivial even dimensional homotopy groups of $SO$ are 
$$ \pi_{8k} SO = \ZZ/2, $$
and the $J$ homomorphism is non-trivial in these degrees \cite{Adams}.  The images of these elements have ASS names 
$$
\im J_{8k} = 
\begin{cases}
\ZZ/2\{ h_3 h_1 \}, & k = 1, \\ 
\ZZ/2\{ P^{k} c_0 \}, & k > 1.
\end{cases}
$$

At the prime $2$ we deduce from the diagram above that there is an exact sequence
$$ 0 \rightarrow (\im J)_{k} \rightarrow \pi_k S \xrightarrow{h_{KO}} \pi_k KO_{(2)}
$$
where $h_{KO}$ is the $KO$-Hurewicz homomorphism.  For positive $k$, the image of $h_{KO}$ is only non-trivial for $k \equiv 1,2 \pmod 8$, where it is isomorphic to $\ZZ/2$ \cite{Adams}. We therefore have for $k \ge 1$ \cite{Mahowaldbo}
\begin{align*}
v_1^{-1}\coker J_{8k+1} = \ZZ/2\{P^{k} h_1 \}, \\
v_1^{-1} \coker J_{8k+2} = \ZZ/2\{P^{k} h_1^2 \}
\end{align*}
where by $v_1^{-1}\coker J$ we mean the cokernel of the composite
$$ \pi_*SO \xrightarrow{J} \pi_*S \rightarrow \pi_* S_{E(1)}. $$
Note that $h_1$ is in the image of $J$, and 
$$ \Phi_K(\{ h_1^2 \}) = 1. $$
We deduce

\begin{prop}
$\Theta_{8k+2} \ne 0$ for $k > 1$.
\end{prop}

Thus we are left to consider dimensions congruent to $-2$ mod $8$, and dimensions congruent to $0$ mod $4$.  As the image of $J$ is detected in $\pi_*S_{E(1)}$, the only classes besides those listed above which can contribute to $\coker J$ in these dimensions must be $v_n$-periodic for $n \ge 2$.

\section{Topological modular forms}\label{sec:tmf}

We saw in the previous section that lifting elements of $\pi_*KO$ to $v_1$-periodic families in $\pi_*^s$ allowed us to deduce the non-triviality of $\coker J$ in certain dimensions.

The bulk of this paper concerns following a similar strategy, where we replace the real $K$-theory spectrum $KO$ with the spectrum of topological modular forms $\tmf$.  In this section we give a brief overview of some important facts regarding this spectrum.  We refer the reader to \cite{Behrenstmf} and \cite{tmf} for more thorough accounts.  

Let $\mc{M}_{\mit{ell}}$ denote the Deligne-Mumford stack of elliptic curves (over $\mr{Spec}(\ZZ)$).  For a commutative ring $R$, the groupoid of $R$-points of $\mc{M}_{\mit{ell}}$ is the groupoid of elliptic curves over $R$.  This stack carries a line bundle $\omega$ where for an elliptic curve $C$, the fiber of $\omega$ over $C$ is given by
$$ \omega_C = T^*_eC, $$
the tangent space of $C$ at its basepoint $e$.

The stack $\mc{M}_{\mit{ell}}$ admits a compactification $\br{\mc{M}}_{\mit{ell}}$ whose $R$ points are generalized elliptic curves (a generalized elliptic curve is a curve over $R$ whose geometric fibers are either elliptic curves or N\'eron $n$-gons) \cite{DeligneRapoport}.  The space of integral modular forms of weight $k$ is defined to be the sections
$$ H^0(\br{\mc{M}}_{\mit{ell}}, \omega^{\otimes k}). $$
We have \cite{Deligne}
$$ H^0(\br{\mc{M}}_{\mit{ell}}, \omega^{\otimes *}) = \ZZ[c_4, c_6, \Delta]/( c_6^2 = c_4^3-1728 \Delta) $$
where $c_k$ has weight $k$.

Goerss-Hopkins-Miller showed that the graded sheaf $\omega^{\otimes *}$ admits a lift to the stable homotopy category.  Namely, they proved that there exists a sheaf $\mc{O}^{\mit{top}}$ of $E_\infty$-ring spectra on the \'etale site $(\br{\mc{M}}_{\mit{ell}})_{et}$ with the property that the spectrum of sections for an \'etale affine open
$$ \mr{Spec}(R) \xrightarrow{C} \br{\mc{M}}_{\mit{ell}} $$
is a spectrum $E_C$ with canonical isomorphisms
\begin{align*}
\pi_{2k}E_C \cong H^0(\mr{spec}(R), \omega^{\otimes k}), \\
\widehat{C} \cong \mr{Spf}{ E^0_C \CC P^\infty}.
\end{align*}
Here, the last isomorphism is between the formal group of the elliptic curve $C$, and the formal group associated to ring spectrum $E_C$.
 
It follows from the fact $\mc{O}^{top}$ is a sheaf in the \'etale topology, and that $\br{\mc{M}}_{ell}$ admits an \'etale cover by affines, that for arbitrary \'etale maps
$$ U \rightarrow \br{\mc{M}}_{\mit{ell}} $$
there is a descent spectral sequence
$$ H^s(U, \omega^{\otimes k}) \Rightarrow \pi_{2k-s} \mc{O}^{top}(U). $$
Motivated by the definition of integral modular forms and this descent spectral sequence in the case of $U = \br{\mc{M}}_{\mit{ell}}$, the spectrum $\Tmf$ is defined to be the global sections 
$$ \Tmf := \mc{O}^{top}(\br{\mc{M}}_{\mit{ell}}). $$

The descent spectral sequence has been computed for $\Tmf$ (see \cite{Konter}).
The spectrum $\Tmf$ is not connective, in part coming from the fact that Serre duality contributes to $H^1(\br{\mc{M}}_{\mit{ell}}, \omega^{\otimes k})$ for $k < 0$, but it also must be noted that there is $2$ and $3$ torsion in $H^s(\br{\mc{M}}_{\mit{ell}}, \omega^{\otimes k})$ for $s \ge 1$ and for both positive and negative values of $k$.  Strangely, differentials wipe out everything in $\pi_t\Tmf$ for $-20 \le t \le -1$.  This phenomenon motivates the definition of the \emph{connective spectrum of topological modular forms} as the connective cover
$$ \tmf := \Tmf\bra{0}. $$ 

There are many other nice features of $\tmf$.  One which we highlight is that the mod $2$ cohomology of $\tmf$ is finitely presented over the Steenrod algebra, and we have \cite{Mathew} 
$$ H^*(\tmf; \FF_2) \cong A\mmod A(2) $$
and thus for any connective spectrum $X$ the mod 2 ASS for $\tmf \wedge X$ takes the form
$$ \Ext^{s,t}_{A(2)_*}(H_*(X;\FF_2)) \Rightarrow \pi_*(\tmf \wedge X)^{\wedge}_2. $$

Also, the modular form $\Delta$ is not a permanent cycle in the descent spectral sequence for $\Tmf$, but $\Delta^{24}$ is, and it turns out that there is an injection 
$$ \pi_* \tmf \hookrightarrow \Delta^{-{24}}\pi_*tmf. $$
Thus $\pi_*\tmf$ is $\abs{\Delta^{24}} = 576$-periodic.  This periodicity improves to $72$-periodicity 3-locally and $192$-periodicity 2-locally (and $24$-periodicity at all other primes).  

\begin{figure}
\includegraphics[angle = 90, origin=c, height =.7\textheight]{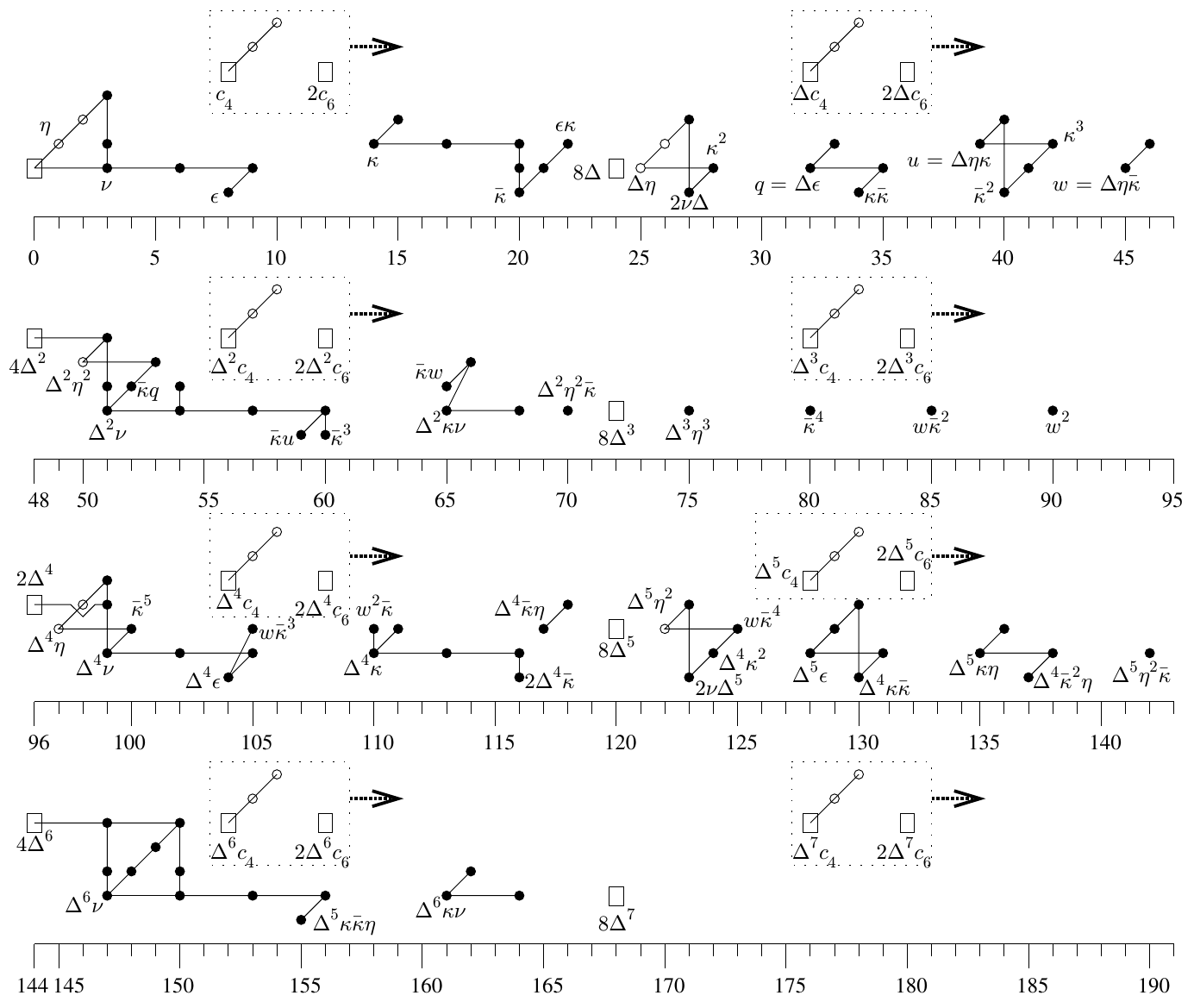}
\caption{The homotopy groups of $\tmf_{(2)}$ ($\pi_*\TMF_{(2)}$ is obtained by inverting $\Delta$)}\label{fig:tmf2}
\end{figure}

For the reader's convenience, we recall the homotopy groups $\pi_*\tmf_{(2)}$ in Figure~\ref{fig:tmf2}.  In this figure:
\begin{itemize}
\item A series of $i$ black dots joined by vertical lines corresponds to a factor of $\ZZ/2^i$ which is annihilated by some power of $c_4$.

\item An open circle corresponds to a factor of $\ZZ/2$ which is not annihilated by a power of $c_4$.

\item A box indicates a factor of $\ZZ_{(2)}$ which is not annihilated by a power of $c_4$.

\item The non-vertical lines indicate multiplication by $\eta$ and $\nu$.

\item A pattern with a dotted box around it and an arrow emanating from the right face indicates this pattern continues indefinitely to the right by $c_4$-multiplication (i.e. tensor the pattern with $\ZZ_{(2)}[c_4]$).
\end{itemize}
After localization at the prime $2$, the element $\Delta^{8}$ is a permanent cycle in the descent spectral sequence, and $\pi_*\tmf_{(2)}$ is given by tensoring the pattern depicted in Figure~\ref{fig:tmf2} with $\ZZ[\Delta^8]$.  Our choice of names for generators in Figure~\ref{fig:tmf2} is motivated by the fact that the elements
$$ \eta, \nu, \epsilon, \kappa, \bar{\kappa}, q, u, w $$
in the 2-primary stable homotopy groups of spheres (see Section~\ref{sec:piS}) map to the corresponding elements in $\pi_*\tmf_{(2)}$ under the $\tmf$-Hurewicz homomorphism.  

In analogy to how $K$ theory has a connective and non-connective variant, we define the \emph{non-connective spectrum of topological modular forms} to be the localization
$$ \TMF := \Delta^{-24} \tmf. $$
There are two things to note about this spectrum.  Firstly, $\TMF$ is the value of the sheaf $\mc{O}^{top}$ on the non-compactified moduli stack of elliptic curves
$$ \TMF \simeq \mc{O}^{top}(\mc{M}_{\mit{ell}}). $$
Secondly, the $\Delta$-periodicity described above is a form of $v_2$-periodicity, which is witnessed by the fact that for any type $2$-spectrum $M$, we have
$$ \TMF \wedge M \simeq (\tmf \wedge M)^f_{E(2)} \simeq (\tmf \wedge M)_{E(2)}. $$

Finally, we note that certain variants of $\TMF$ can be constructed associated to congruence subgroups of $SL_2(\ZZ)$.  Specifically, associated to the congruence subgroup
$$ \Gamma_0(\ell) := \{ A \in SL_2(\ZZ) \: : \: \text{$A$ is upper triangular mod $\ell$} \} $$
there is an \'etale cover
$$ f: \mc{M}(\Gamma_0(\ell)) \rightarrow \mc{M}_{\mit{ell}}[1/\ell] $$
and an associated spectrum (see \cite{Behrens})
$$ \TMF_0(\ell) := \mc{O}^{top}(\mc{M}(\Gamma_0(\ell))). $$
Geometrically, the moduli stack $\mc{M}(\Gamma_0(\ell))$ is the moduli stack of pairs $(C,H)$ where $C$ is an elliptic curve and $H \le C$ is a cyclic subgroup of order $\ell$, and the map $f$ is defined on $R$-points by forgetting the level structure:
$$ f_{\ell}: (C,H) \mapsto C. $$
The functoriality of $\mc{O}^{top}$ associates to the map $f$ is a map of $E_{\infty}$-ring spectra
$$ f_{\ell}^*: \TMF \rightarrow \TMF_0(\ell). $$
There is another \'etale map
$$ q_\ell : \mc{M}(\Gamma_0(\ell)) \rightarrow \mc{M}_{\mit{ell}} $$
which on $R$-points is given by
$$ q_\ell: (C,H) \mapsto C/H. $$
Associated to this map is an $E_\infty$-ring map
$$ q_\ell^*: \TMF \rightarrow \TMF_0(\ell) $$
which serves as a kind of generalization of the classical $\ell$th Adams operations to $\TMF$.
These operations, and the generalizations of the $J$-spectrum which may be constructed using these, are studied in \cite{Behrens}, \cite{MahowaldRezk}, and \cite{BO}.

\section{The strategy for lifting elements from $\tmf_*$}\label{sec:strategy}

For the remainder of this paper (until Section~\ref{sec:cokerJ}) we will be working $2$-locally, and homology will be implicitly taken with mod 2 coefficients.  In this section we outline our strategy to lift elements from $\pi_*\tmf$ to $\pi_*^s$.  Namely, given a $v_2$-periodic element 
$x \in  \tmf_*$, we will lift it to
$$ \td{x} \in \tmf_* M(2^i, v_1^j) $$
so that the projection to the top cell maps $\td{x}$ to $x$.  We will then show, using the $\tmf$-resolution, that $\td{x}$ lifts to an element 
$$ \td{y} \in \pi_*M(2^i, v_1^j). $$ 
Then the image 
$$ y \in \pi_*^s $$
given by projecting $\td{y}$ to the top cell is an element whose image under the $\tmf$-Hurewicz homomorphism is $x$.  \emph{The element $y$ is $v_2$-periodic because its image in $\tmf$ is $v_2$-periodic}.

The first major observation is the following.

\begin{prop}
Every $v_2$-periodic element $x \in \pi_{*}\tmf$ is $8$-torsion and $v_1^8$-torsion.
\end{prop}

\begin{proof}
Being $v_2$-periodic is equivalent to lifting to 
$$ \pi_* \tmf \wedge M(2^i, v_1^j) $$
for some values of $i$ and $j$.
Since in $\tmf_*M(8)$ we have $c_4 = v_1^4$ (this can be seen by combining the equality $v_1 = a_1$ at $p = 2$ in the proof of Lemma~21.4 with the equation for $c_4$ in the proof of Proposition~18.7 in \cite{Rezk}), the proposition is equivalently stated as asserting that every $c_4$-torsion class in $\pi_* \tmf$ is $c_4^2$-torsion and $8$-torsion.  Using the fact that $c_4^2$ has Adams filtration $8$, this is easily checked from the Adams $E_\infty$ page for $\tmf$ (see, for instance, \cite[pp. 196-7]{tmf}).
\end{proof}

We therefore apply the strategy above to lift $v_2$-periodic elements in $\pi_*\tmf$ to the top cell of $M(8,v_1^8)$, and endeavor to lift these to homotopy classes.  The lifts described above will be performed by analyzing the \emph{modified Adams spectral sequence (MASS)} for $M(8,v_1^8)$ (see \cite[Sec.~3]{BHHM}).  

\subsection*{The modified Adams spectral sequence}

We explain how to alter the discussion of \cite[Sec.~3]{BHHM} 
to construct the \emph{modified Adams spectral sequence (MASS)} for $M(8,v_1^8)$.

The problem with the classical ASS is that the Steenrod algebra acts trivially on the cohomology $H^*(M(8, v_1^8); \FF_2)$, and thus the $E_2$ term
$$ \Ext^{*,*}_{A_*}(H_*(M(8, v_1^8);\FF_2)) $$
is just a direct sum of four copies of the $E_2$-term for the sphere, and does not recognize the non-trivial attaching maps between these cells.  This is because the degree $8$ map on the sphere has Adams filtration $3$ and the self-map $v_1^8$ has Adams filtration $8$.  The MASS corrects for this by taking these higher Adams filtrations into account.  The resulting spectral sequence will have a vanishing line of slope less than the slope  of the vanishing line for the classical ASS.
The lifts of our desired classes to $M(8,v_1^8)$ will be located near the vanishing line, which will translate to less possible targets for differentials.

Let 
$$
\xymatrix{
S \ar@{=}[r] & S_0 \ar[d] & S_1 \ar[l] \ar[d] & S_2 \ar[d]
\ar[l] & \cdots \ar[l] \\
& K_0 & K_1 & K_2 
}
$$
denote the canonical Adams resolution of the sphere,
where
\begin{align*}
S_i & = \br{H}^{\wedge i}, \\
K_i & = H \wedge S_i.
\end{align*}
Here $H$ denotes the Eilenberg-MacLane spectrum $H\FF_2$, and $\br{H}$
denotes the fiber of the unit $S \rightarrow H$.

Since the map $8: S \rightarrow S$ has Adams
filtration $3$, there exists a lift\footnote{Of course the map may be lifted to $\td{S}_3$, but this turns out to result in a less intuitive variant (e.g. we do not need to modify the ASS at all for $M(2)$, and the degree 2 map has Adams filtration 1).}:
$$
\xymatrix{
& S_2 \ar[d] \\
S \ar@{.>}[ru]^{\td{8}} \ar[r]_{8} & S
}
$$
The lift $\td{8}$ induces a map of Adams resolutions:

\begin{equation}\label{diag:aresmap}
\xymatrix{
S_0 \ar[d]_{8} &  S_0 \ar@{=}[l] \ar[d]_{(\td{8})_1} &  S_0 \ar@{=}[l]
\ar[d]_{(\td{8})_2} & S_1 \ar[l]
\ar[d]_{(\td{8})_{3}} & \cdots \ar[l] \\
S_0 &  S_1 \ar[l] &  S_2 \ar[l] &  S_{3}\ar[l] & \cdots \ar[l]
}
\end{equation}
where the maps $(\td{8})_i$ are given by
$$ (\td{8})_i : S_{i-2} = \br{H}^{\wedge i-2} \xrightarrow{1 \wedge \td{8}} \br{H}^{\wedge i-2} \wedge
\br{H}^{\wedge 2} = S_{i}.
$$
The mapping cones of the vertical maps of (\ref{diag:aresmap})
$$ S_{i-2} \xrightarrow{(\td{8})_i} S_i \rightarrow
M(8)_i $$
form a resolution of $M(8)$:
\begin{equation}\label{eq:MASSM3}
 \xymatrix{
 M(8) \ar@{=}[r] & M(8)_0 \ar[d] & M(8)_1 \ar[l] \ar[d] 
 & \ar[l] M(8)_2 \ar[d] & \ar[l] \cdots \\
 & K(8)_0 & K(8)_1 & K(8)_2 
 }
 \end{equation}
One can get an intuitive understanding of the spectra $M(8)_i$ by using mapping cones and telescopes to regard each of the $S_i's$ as a subcomplex of $S_0 \simeq S$.  For the $n$-disk $D^n$, we define a corresponding sequence of subcomplexes of $S_0 \wedge D^n$
$$ D^n_i := \begin{cases}
S_i \wedge D^n, & i \ge 0, \\
S_0 \wedge D^n, & i < 0.
\end{cases}
$$
In this language we may regard $M(8)_i$ as the CW-spectrum
$$ M(8)_i = S_i \cup_{\td{8}} D^1_{i-2}. $$

The \emph{modified Adams spectral sequence} (MASS) for $M(8)$ is the spectral sequence resulting from applying $\pi_*$ to the resolution (\ref{eq:MASSM3}):
$$ E_1^{s,t}(M(8)) = \pi_{t-s}(K(8)_s) \Rightarrow \pi_{t-s}M(8).
$$
We now describe the $E_2$-term of this MASS.
Let $\mc{D}_{A_*}$
denote the derived category of $A_*$-comodules.  
For objects $M$ and $N$
of $\mc{D}_{A_*}$, we define groups
$$ \Ext^{s,t}_{\Gamma}(M,N) = \mc{D}_{\Gamma}(\Sigma^t M, N[s]) $$
as a group of maps in the derived category.  Here $\Sigma^t M$ denotes the
$t$-fold shift with respect to the internal grading of $M$, and $N[s]$ denotes
the $s$-fold shift with respect to the triangulated structure of
$\mc{D}_{A_*}$.  This
reduces to the usual definition of $\Ext_{A_*}$ when $M$ and $N$ are
$A_*$-comodules.  

Define $H(8)$ to be the cofiber of the map
\begin{equation}\label{eq:H3}
 \Sigma^3 \FF_2[-3] \xrightarrow{h_0^3} \FF_2
 \end{equation}
in the derived category of $A_*$-comodules.
Using precisely the same arguments as in \cite[Sec.~3]{BHHM}, we find the $E_2$-term of the MASS takes the form
$$
E_2^{s,t}(M(8)) = \Ext^{s,t}_{A_*}(H(8)) 
\Rightarrow
\pi_{t-s}(M(8)),
$$

We now endeavor to give a similarly modified Adams resolution of
$$ M(8,v_1^8) = S^0 \cup_{8} D^1 \cup_{\{h_4h_0^5\}} D^{17} \cup_{8} D^{18} $$
We use the fact that $\{h_4h_0^5\}$ has Adams filtration $6$ to make the lift 
$$
\xymatrix{
& S_{5}^{0} \ar[d] \\
S^{15} \ar@{.>}[ru]^{\td{\{h_4h_0^5\}}} \ar[r]_{\{h_4h_0^5\}} & S^0
}
$$
Because $h_4h_0^8 = 0$ in the $E_2$-page of the Adams spectral sequence, the composite $ (\td{8})_7 \td{\{h_4h_0^5\}}$ is null, and we can therefore make the lift
$$ 
\xymatrix{
& S_0^{16} \ar[d]^{\{h_4h_0^5\}} \ar@{.>}[dl] \\
M(8)_7 \ar[r] & S^1_5 \ar[r]_{(\td{8})_7} & S^1_7.
}
$$
We deduce the existence of the CW-spectrum
$$ M(8)_7 \cup_{\td{\{h_4h_0^5\}}} D_0^{17}. $$ 
Finally, we use the fact that the class $h_4h_0^8[1]$ is zero in the $E_2$-term of the MASS for $M(8)$ to deduce the existence of a lift
$$ 
\xymatrix{
& S_0^{17} \ar[d]^{\td{8}} \ar@{.>}[dl] \\
M(8)_9 \cup_{(\td{h_4h_0^5})_2} D_2^{17} \ar[r] & S^{17}_2 \ar[r]_{(\td{h_4h_0^5})_2[1]} & \Sigma M(8)_9.
}
$$
and hence the existence of the CW spectra
$$ M(8,v_1^8)_i := S^0_i \cup_{\td{8}} D^1_{i-2} \cup_{\td{h_4h_0^5}} D^{17}_{i-7} \cup_{\td{8}} D^{19}_{i-9}. $$

These give a resolution:
$$
\xymatrix{
M(8,v_1^8) \ar@{=}[r] & M(8,v_1^8)_0 \ar[d] & M(8,v_1^8)_1 \ar[l] \ar[d] 
& \ar[l] M(8,v_1^8)_2 \ar[d] & \ar[l] \cdots \\
& K(8,v_1^8)_0 & K(8,v_1^8)_1 & K(8,v_1^8)_2 
}
$$
Smashing this resolution with $X$, we obtain the \emph{modified Adams spectral sequence for $M(8,v_1^8) \wedge X$}:
\begin{equation}\label{eq:MASS1}
E_1^{s,t}(M(8,v_1^8) \wedge X) = \pi_{t-s}(K(8,v_1^8)_s \wedge X) \Rightarrow \pi_{t-s}(M(8,v_1^8)
\wedge X).
\end{equation}
Define $H(8,v_1^8) \in \mc{D}_{A_*}$ to be the cofiber
\begin{equation}\label{eq:H38}
\Sigma^{24}H(8)[-8] \xrightarrow{v_1^8} H(8) \rightarrow H(8,v_1^8).
\end{equation}
Again using the methods of \cite[Sec.~3]{BHHM},
we deduce that the $E_2$-term of the spectral sequence (\ref{eq:MASS1}) is
given by
$$ E_2^{s,t}(M(8,v_1^8) \wedge X) = \Ext^{s,t}_{A_*}(H(8,v_1^8)\otimes H_*X). $$
The $E_2$-term of this spectral sequence is readily computed by using the pair of long exact sequences coming from the triangles (\ref{eq:H3}) and (\ref{eq:H38}).
Taking $X = \tmf \wedge Y$ for some $Y$, and applying a change of rings theorem, the MASS takes the form
$$ E_2^{s,t}(\tmf \wedge M(8,v_1^8) \wedge Y) = \Ext^{s,t}_{A(2)_*}(H(8,v_1^8) \otimes H_* Y) \Rightarrow \tmf_{t-s}(M(8,v_1^8) \wedge Y). $$

\begin{figure}
\centering
\includegraphics[height=.5\textheight]{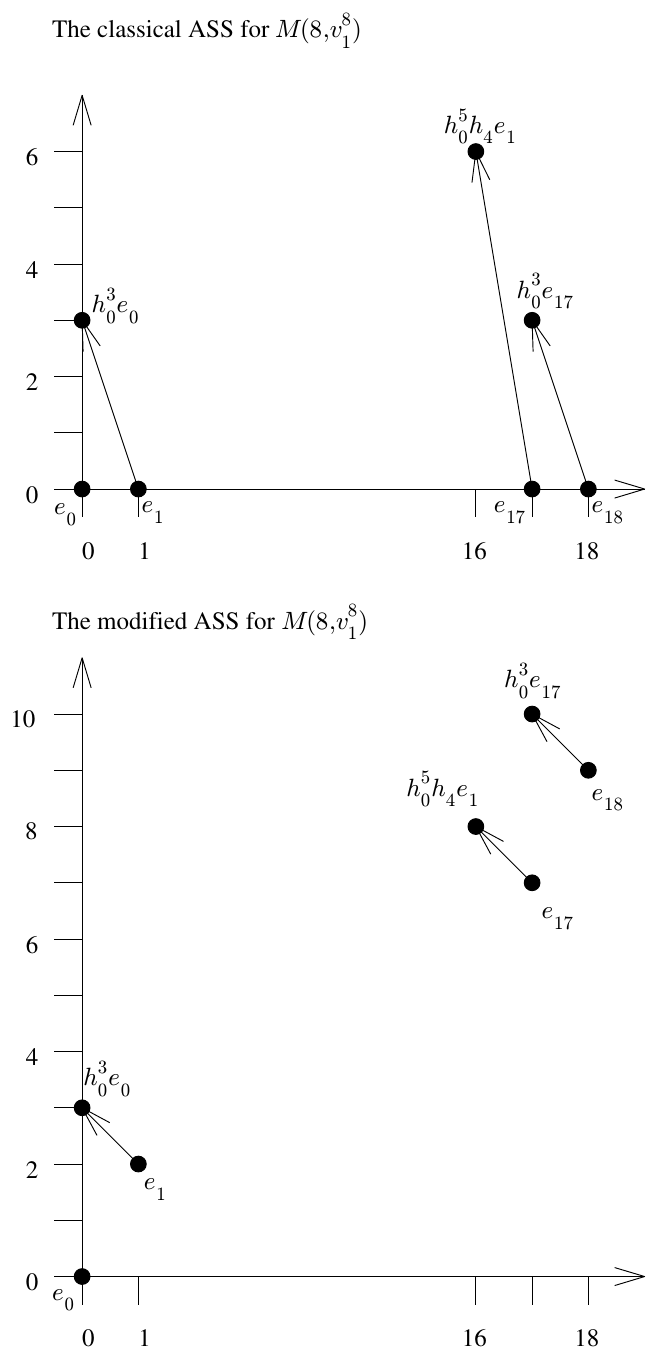}
\caption{Comparison of the ASS and the MASS for $M(8,v_1^8)$.}
\label{fig:ASSvsMASS}
\end{figure}

Figure~\ref{fig:ASSvsMASS} illustrates the difference between the ASS and the MASS for $M(8,v_1^8)$.  The classes $e_i$ correspond to the different cells in $M(8,v_1^8)$.  In the ASS, the differentials on the classes $e_i$ induced from the attaching maps in $M(8,v_1^8)$ are long differentials, wheras in the MASS the same differentials are $d_1$-differentials.  Thus, the attaching maps in $M(8,v_1^8)$ result in a much smaller $E_2$-page.

\subsection*{The algebraic $\tmf$-resolution}

The $E_2$-page of the MASS for $M(8,v_1^8)$ will be analyzed using an algebraic analog of the $\tmf$-resolution (as in \cite[Sec.~5]{BHHM}).  

The (topological) $\tmf$-resolution of a space $X$ is the Adams resolution based on the spectrum $\tmf$:
$$
\scriptsize{
\xymatrix@C-1em{
X\ar[d] &
\ar[l] \br{\tmf} \wedge X \ar[d] &
\ar[l] \br{\tmf}^{\wedge 2} \wedge X \ar[d] &
\cdots \ar[l]
\\
\tmf \wedge X & 
\tmf \wedge \br{\tmf} \wedge X & 
\tmf \wedge \br{\tmf}^{\wedge 2} \wedge X & 
}}
$$
Here, $\br{\tmf}$ is the fiber of the unit:
$$ \br{\tmf} \rightarrow S \rightarrow \tmf. $$
The resulting spectral sequence takes the form
$$ E_1^{s,t} = \pi_{t-s} \tmf \wedge \br{\tmf}^{s} \wedge X \Rightarrow \pi_{t-s} X. $$

The \emph{algebraic} $\tmf$-resolution is an algebraic analog.
Namely, for any object $M$ of the derived category of $A_*$-comodules, we apply $\Ext_{A_*}(-)$ to the following diagram in the derived category:
$$
\scriptsize{
\xymatrix@C-1em{
M\ar[d] &
\ar[l] \br{A\mmod A(2)}_*[-1] \otimes M \ar[d] &
\ar[l] \br{A\mmod A(2)}^{\otimes 2}_*[-2] \otimes M \ar[d] &
\cdots \ar[l]
\\
A\mmod A(2)_* \otimes M & 
A \mmod A(2)_* \otimes \br{A\mmod A(2)}_*[-1] \otimes M & 
A \mmod A(2)_* \otimes \br{A\mmod A(2)}^{\otimes 2}_*[-2] \otimes M & 
}}
$$
Here $\br{A\mmod A(2)}_*$ is the cokernel of the unit
$$ 0 \rightarrow \FF_2 \rightarrow A\mmod A(2)_* \rightarrow \br{A \mmod A(2)}_* \rightarrow 0 $$
(note that $H_*\Sigma \br{\tmf} = \br{A\mmod A(2)}_*$).
This results in the \emph{algebraic $\tmf$-resolution} 
$$ E_1^{s,t,n} = \Ext^{s,t}_{A(2)_*}(\br{A\mmod A(2)}_*^{\otimes n} \otimes M) \Rightarrow \Ext^{s+n,t}_{A_*}(M). $$

\begin{rmk}
The algebraic $\tmf$-resolution is related to the topological $\tmf$-resolution by Adams spectral sequences:
$$ \Ext^{s,t}_{A(2)_*}(\br{A\mmod A(2)}_*^{\otimes n} \otimes H_*X) \Rightarrow \pi_{t-s-n}\tmf \wedge \br{\tmf}^{\wedge n} \wedge X. $$ 
\end{rmk}

\subsection*{The implementation of the strategy}

We now return to the strategy outlined at the beginning of this section, and explain how it is implemented in this paper using the MASS and algebraic $\tmf$-resolution.  Given an element $x \in \tmf_*$, we want to lift it to an element $y \in \pi_*^s$.  To this end, we consider the diagram of (M)ASS's:
$$
\xymatrix{
& \Ext_{A(2)_*}(H(8,v_1^8)) \ar@{=>}[rr] \ar[dd]
&& \tmf_*M(8,v_1^8) \ar[dd] \\
\Ext_{A_*}(H(8,v_1^8)) \ar@{=>}[rr] \ar[ru] \ar[dd]  
&& \pi_* M(8,v_1^8) \ar[ru] \ar[dd] & \\
& \Ext_{A(2)_*}(\FF_2) \ar@{=>}[rr] 
&& \tmf_* \\
\Ext_{A_*}(\FF_2) \ar@{=>}[rr] \ar[ru]
&& \pi_*^s \ar[ru]
}
$$ 
First, we identify an element
$$ [x] \in \Ext_{A(2)_*}(\FF_2) $$
which detects the element $x$ in the ASS for $\tmf_*$, and then we identify an element 
$$ [\td{x}] \in \Ext_{A(2)_*}(H(8,v_1^8)) $$
which maps to it.  This element $[\td{x}]$ can be regarded as an element of the zero line of the algebraic $\tmf$-resolution for $\Ext_{A_*}(H(8,v_1^8))$.  We will show (Proposition~\ref{prop:tmfresPC}) that the element $[\td{x}]$ is a permanent cycle in the algebraic $\tmf$-resolution, and thus lifts to an element
$$ [\td{y}] \in \Ext_{A_*}(H(8,v_1^8)). $$
We will then show (Theorem~\ref{thm:MASSPCs}) that the element $[\td{y}]$ is a permanent cycle in the MASS for $M(8,v_1^8)$, and hence detects an element
$$ y \in \pi_*M(8,v_1^8). $$
It then follows that the image of $y$ in $\tmf_*$ equals $x$, modulo terms of higher Adams filtration.

\section{Useful facts about $M(8,v_1^8)$}\label{sec:M38}

This section collects some facts about $M(8,v_1^8)$ which will prove useful.  For us, a \emph{weak homotopy ring spectrum} is a spectrum $R$ with a unit 
$$ S \xrightarrow{u} R $$
and a multiplication map
$$ R\wedge R \xrightarrow{\mu} R $$
so that the two composites
\begin{gather*}
R \wedge S \xrightarrow{u \wedge 1} R \wedge R \xrightarrow{\mu} R, \\
S \wedge R \xrightarrow{1 \wedge u} R \wedge R \xrightarrow{\mu} R
\end{gather*}
are equivalences.

\begin{prop}
$M(8,v_1^8)$ is a weak homotopy ring spectrum.
\end{prop}

\begin{proof}
We duplicate the argument given in \cite{Mahowaldring} in the case of $M(8,v_1^8)$.  A similar argument in fact shows that $M(8,v_1^{4t})$ is a ring spectrum for all $t > 0$.  The unit
$$ u: S^0 \rightarrow M(8,v_1^8) $$
is the inclusion of the bottom cell.
We will show there exists an extension
\begin{equation}\label{eq:extension}
\xymatrix{
S^0 \ar[r] \ar[d] & M(8,v_1^8) \\
M(8,v_1^8) \wedge M(8,v_1^8) \ar@{.>}[ru]_{\mu}
}
\end{equation}
From this it follows that the composites
\begin{gather*}
M(8,v_1^8) \wedge S \xrightarrow{u \wedge 1} M(8,v_1^8) \wedge M(8,v_1^8) \xrightarrow{\mu} M(8,v_1^8), \\
S \wedge M(8,v_1^8) \xrightarrow{1 \wedge u} M(8,v_1^8) \wedge M(8,v_1^8) \xrightarrow{\mu} M(8,v_1^8)
\end{gather*}
are equivalences, since these composites are the identity on the bottom cell, and $BP_*M(8,v_1^8)$ is generated as a $BP_*$-module on the generator coming from the bottom cell.

An elaboration of the technique used to construct the MASS for $M(8,v_1^8)$ gives a MASS
\begin{equation}\label{eq:MASSh382}
 \Ext^{s,t}_{A_*}(H(8,v_1^8)^{\otimes 2}, H(8,v_1^8)) \Rightarrow [\Sigma^{t-s} M(8,v_1^8)^{\wedge 2}, M(8,v_1^8)].
 \end{equation}
We shall say a cell $\Sigma^t\FF_2[-s]$ in the derived category of $A_*$-comodules has bidgeree $(t-s,s)$.
In the derived category of $A_*$-comodules, $H(8,v_1^8)^{\otimes 2}$ has cells in the bidegrees
\begin{equation}\label{eq:bideglist}
 (0,0), \: (1,2), \: (2,4), \: (17,7), \:(18,9), \:(19,11), \:(34,14), \:(35,16), \:(36,18).
\end{equation}
Applying $\Ext^{*,*}_{A_*}(-,H(8,v_1^8))$ to the cellular filtration of $H(8,v_1^8)^{\otimes 2}$ in $\mc{D}_{A_*}$, we get an \emph{algebraic Atiyah-Hirzebruch spectral sequence} 
$$ \bigoplus_{(t_0-s_0, s_0)} \Ext^{s-s_0,t-t_0}_{A_*}(H(8,v_1^8)) \Rightarrow \Ext^{s,t}_{A_*}(H(8,v_1^8)^{\otimes 2}, H(8,v_1^8)) $$
where $(t_0-s_0, s_0)$ ranges through the list (\ref{eq:bideglist}). The extension problem (\ref{eq:extension}) is equivalent to showing that the class
$$ 1[0,0] \in \Ext^{0,0}_{A_*}(H(8,v_1^8)) $$
supported on the bottom cell of $H(8,v_1^8)^{\otimes 2}$
is a permanent cycle in both the algebraic Atiyah-Hirzebruch spectral sequence, and the MASS.  The possible targets of a differential (in either spectral sequence) supported by this class are detected in the algebraic Atiyah-Hirzebruch spectral sequence by classes in 
$$ \Ext^{s_0+r, (k_0 - 1) + (s_0 + r)}_{A_*}(H(8,v_1^8)) $$
for $r \ge 1$ and $(k_0, s_0)$ in the list (\ref{eq:bideglist}).  One can check these groups are all zero. This check can be performed by hand, or alternatively, in this range the groups are isomorphic to
$$ \Ext^{s_0+r, (k_0 - 1) + (s_0 + r)}_{A(2)_*}(H(8,v_1^8)) $$
and these latter groups are displayed in Figure~\ref{fig:ExtH38}.
\end{proof}

We will make use of the following corollary, which implies that the differentials satisfy a Liebnetz rule, and that products can be computed up to higher filtration in these spectral sequences.

\begin{cor}
Both the algebraic $\tmf$-resolution and the MASS for $M(8,v_1^8)$ are spectral sequences of (possibly non-associative) algebras.
\end{cor}

\begin{lem}\label{lem:v2^8}
The element
$$ v_2^8 \in \Ext^{8,48+8}_{A(2)_*}(H(8,v_1^8)) $$
is a permanent cycle in the algebraic $\tmf$-resolution, and gives rise to an element
$$ v_2^8 \in \Ext^{8,48+8}_{A_*}(H(8,v_1^8)). $$
\end{lem}

\begin{proof}
By way of motivation, we point out that in the May spectral sequence for $\Ext_{A_*}(\FF_2)$ there is a differential
$$ d_8(b_{3,0}^4) = b_{2,0}^4h_5. $$
We therefore endeavor to construct $v_2^8$ as a lift of
$$ h_5[0] \in \Ext^{1,31+1}_{A_*}(H(8)) $$
in the long exact sequence
$$ \Ext^{8,48+8}_{A_*}(H(8,v_1^8)) \xrightarrow{\partial} \Ext^{1,31+1}_{A_*}(H(8)) \xrightarrow{v_1^8} \Ext^{9,47+9}_{A_*}(H(8)). $$
Here and throughout this proof, we use $x[0]$ (respectively $x[1]$) to denote a class in $\Ext_{A_*}(H(8))$ corresponding to $x \in \Ext_{A_*}(\FF_2)$ supported on the 0-cell (respectively 1-cell) of $H(8)$.
We must show that 
$$ v_1^8 \cdot h_5[0] = 0 \in \Ext^{9,47+9}(H(8)). $$
The only other possibility is that $v_1^8 \cdot h_5[0] = B_1[1]$.  If that were the case, then the image of $h_5$ under the composite
$$ \Ext^{1,31+1}_{A_*}(\FF_2) \rightarrow \Ext^{1,31+1}_{A_*}(H(8)) \xrightarrow{v_1^8} \Ext^{9,47+9}_{A_*}(H(8)) \rightarrow \Ext^{7, 46+7}_{A_*}(\FF_2) $$
would be $B_1$.   Since $h_4 h_5 = 0$, this would imply that $h_4B_1 = 0$.  However, according to Bruner's computer calculations, $h_4B_1 \ne 0$ \cite{BrunerExt}.  We deduce that $v_1^8 \cdot h_5[0]$ is indeed zero.

Let 
$$ h_5[17] \in \Ext^{8,48+8}_{A_*}(H(8,v_1^8)) $$
be the resulting lift of $h_5[0]$.
It remains to show is that $h_5[17]$ is detected by 
$$ v_2^8 \in \Ext^{8,48+8}_{A(2)_*}(H(8,v_1^8))$$ 
in the algebraic $\tmf$-resolution.  Consider the fiber
$$ F \rightarrow A \mmod A(2)_* \rightarrow A \mmod A(2) \otimes \br{A \mmod A(2)}_* $$
in the derived category of $A_*$-comodules.
The associated long exact sequence of Ext groups is the first two lines of the algebraic tmf-resolution:
$$ \cdots \rightarrow \Ext_{A_*}(F \otimes M) \rightarrow \Ext_{A(2)_*}(M) \xrightarrow{d^M_1} \Ext_{A(2)_*}(\br{A \mmod A(2)}_*\otimes M) \rightarrow \cdots. $$

We first wish to show that
$$ d_1^{H(8)}(v_2^8) = v_1^8 \xib_1^{32} $$
(here we are regarding $\xib_1^{32}$ as an element of $\br{A\mmod A(2)}_*$).  Using the fact that $v_2^8$ corresponds to the modular form $\Delta^2$, in the language of \cite[Sec.~6]{BOSS}, $d_1(\Delta^2)$ corresponds to the 2-variable modular form:
$$ d_1(\Delta^2) = \Delta(q)^2 - \Delta(\bar{q})^2. $$
Using a computer, in terms of the generators for the ring of 2-variable modular forms produced in \cite[Table~1]{BOSS}, we find
\begin{multline*}
\Delta(q)^2 - \Delta(\bar{q})^2 \equiv 400f_1c_4^2\Delta + 288f_5 \Delta + 292f_4 \Delta + 341f_1^2 c_4^4 + 108f_5 c_4^3+ 190f_4 c_4^3 \\
+ 436f_1^3c_4^3 + 278f_9 c_4^2 + 364f_1 f_5 c_4^2 + 367f_1^4 c_4^2 + 32f_{11}c_4 + 110f_8 c_4 + 352f_1^2 f_5 c_4 + 384f_1^5c_4 \\ + 192f_2 f_7 + 256f_5^2 \quad (\text{mod $512$})
\end{multline*}
and thus (consulting the Adams filtrations of \cite[Table~1]{BOSS}) we have
$$ d_1(\Delta^2) = c_4^2 f_1^4 + 256 f_5^2 + \text{terms of higher Adams filtration}. $$ 
Using the fact (again from \cite[Table~1]{BOSS}) that $f_1^4$ and $f_5^2$ are detected in $\Ext_{A(2)_*}(\br{A \mmod A(2)}_*)$ by $\zeta_1^{32}$ and $\zeta_2^{16}$ (and that $c_4$ corresponds to $v_1^4$), we deduce 
$$ d_1^{\FF_2}(v_2^8) = v_1^8 \zeta_1^{32} + h_0^8 \zeta_2^{16}. $$
It follows that
$$ d_1^{H(8)}(v_2^8) = v_1^8 \zeta_1^{32} $$
and therefore
$$ d_1^{H(8,v_1^8)}(v_2^8) = 0. $$
Therefore $v_2^8 \in \Ext_{A(2)_*}$ lifts to an element
$$ \td{v_2^8} \in \Ext_{A_*}(F \otimes H(8,v_1^8)). $$
The geometric boundary theorem applied to the triangle of truncated algebraic tmf-resolutions
$$
\scriptsize
\xymatrix@C-2em{
\quad \ar[r]^-{\partial^F} &
\Ext_{A_*}(F \otimes \Sigma^{24}H(8)[-8]) \ar[r]^{v_1^8} \ar[d] & \Ext_{A_*}(F\otimes H(8)) \ar[r] \ar[d] &
\Ext_{A_*}(F \otimes H(8,v_1^8)) \ar[d] 
\\
\quad \ar[r]  &
\Ext_{A(2)_*}(\Sigma^{24}H(8)[-8]) \ar[r]^{v_1^8} \ar[d]^{d_1^{H(8)}} & 
\Ext_{A(2)_*}(H(8)) \ar[r] \ar[d]^{d_1^{H(8)}} &
\Ext_{A(2)_*}(H(8,v_1^8)) \ar[d]^{d_1^{H(8,v_1^8)}} 
\\
\quad \ar[r] &
\Ext_{A(2)_*}(\br{A\mmod A(2)}_* \otimes \Sigma^{24}H(8)[-8]) \ar[r]^-{v_1^8} & 
\Ext_{A(2)_*}(\br{A\mmod A(2)}_* \otimes H(8)) \ar[r]  &
\Ext_{A(2)_*}(\br{A\mmod A(2)}_* \otimes H(8,v_1^8)) 
}
$$
allows us to deduce that
$\partial^F(\td{v_2^8}) = \iota_2(h_5[0])$
in the following diagram:
$$
\xymatrix{
&\Ext_{A_*}(H(8,v_1^8)) \ar[r]^-\partial \ar[d]_{\iota_1} & \Ext_{A_*}(\Sigma^{24}H(8)[-8]) \ar[d]^{\iota_2} \\ 
\Ext_{A_*}(F \otimes H(8)) \ar[r] &  \Ext_{A_*}(F \otimes H(8,v_1^8)) \ar[r]_-{\partial^F} \ar[d] & 
\Ext_{A_*}(F \otimes \Sigma^{24}H(8)[-8])  \\
& \Ext_{A(2)_*}(H(8,v_1^8)) &
}
$$
If we can show that 
$$ \iota_1(h_5[17]) = \td{v_2^8} $$
then the proof will be complete (since $\td{v_2^8}$ maps to $v_2^8$ by construction).  Since $\partial (h_5[17]) = h_5[0]$,  we deduce that 
$$ \partial^F(\iota_1(h_5[17])-\td{v_2^8}) = 0 $$
and therefore $\iota_1(h_5[17])-\td{v_2^8}$ lifts to an element 
$$ y \in \Ext^{8,48+8}_{A_*}(F \otimes H(8)). $$  
The proof will be completed by showing this Ext group is zero.  Considering the truncated algebraic $\tmf$-resolution for $F \otimes H(8)$, the only non-trivial element in
$$ \Ext^{8,48+8}_{A(2)_*}(H(8)) $$
is $v_2^8$, and we have already established that $d_1^{H(8)}(v_2^8) \ne 0$.  Examination of the Ext charts in \cite[Sec.~5]{BOSS} reveals that
$$ \Ext^{7,49+7}_{A(2)_*}(\br{A \mmod A(2)}_*\otimes H(8)) = 0. $$
We deduce that $\Ext^{8,48+8}_{A_*}(F \otimes H(8)) = 0$, as desired.
\end{proof}

\section{$\bo$-Brown-Gitler comodules}\label{sec:boi}

Brown and Gitler constructed spectra (Brown-Gitler spectra) which realize certain quotients of the Steenrod algebra (Brown-Gitler modules) \cite{BrownGitler}.  These have found a variety of important applications in manifold theory and homotopy theory.  An integral variant of their construction (integral Brown-Gitler modules/spectra) were found by Mahowald to play an important role in the theory of $\bo$-resolutions \cite{Mahowaldbo}.  Mahowald, Jones, and Goerss constructed analogous modules/spectra in the context of connective $K$-theory \cite{GoerssJonesMahowald}, and these play a role in the theory of $\tmf$-resolutions analogous to the role that integral Brown-Gitler modules/spectra play in the $\bo$-resolution \cite{MahowaldRezk}, \cite{DavisMahowald}, \cite{BHHM}, \cite{BOSS}.  

For the purposes of this paper we shall be concerned with the duals of the $\bo$-Brown-Gitler modules, which are $A_*$-comodules.  Endow the mod $2$ homology of the connective real $K$-theory spectrum 
$$ H_*(\bo;\FF_2) \cong A\mmod A(1)_* = \FF_2[\xib_1^4, \xib_2^2, \xib_3, \ldots] $$
with an grading by declaring the \emph{weight} of $\xib_i$ to be $2^{i-1}$.
The $i$th \emph{$\bo$-Brown-Gitler} comodule is the subcomodule
$$ \bou_i = F_{4i}A\mmod A(1)_* \subset A \mmod A(1)_* $$
spanned by elements of weight less than or equal to $4i$.

The analysis of the $E_1$-page of the algebraic $\tmf$-resolution is simplified via the decomposition of $A(2)_*$-comodules
$$ \br{A \mmod A(2)}_* \cong \bigoplus_{i > 0} \Sigma^{8i} \bou_i $$
of \cite[Cor.~5.5]{BHHM}.
We therefore have a decomposition of the $E_1$-page of the algebraic $\tmf$-resolution for $M$ given by
\begin{equation}\label{eq:E1decomp}
E_1^{s,t,n} \cong \bigoplus_{i_1, \ldots, i_n > 0}\Ext_{A(2)_*}^{s,t}(\Sigma^{8(i_1+\cdots+i_n)}\bou_{i_1} \otimes \cdots \otimes \bou_{i_n} \otimes M).
\end{equation}

For any $M$, the computation of
$$ \Ext_{A(2)_*}^{s,t}(\Sigma^{8(i_1+\cdots+i_n)}\bou_{i_1} \otimes \cdots \otimes \bou_{i_n} \otimes M) $$
can be inductively determined from $\Ext_{A(2)_*}(\bou_1^{\otimes k} \otimes M)$ by means of a set of exact sequences of $A(2)_*$-comodules which relate the $\bou_i$'s \cite[Sec.~7]{BHHM} (see also \cite{BOSS}):
\begin{gather}
0\to \Sigma^{8j} \ul{\bo}_j \to \ul{\bo}_{2j}\to A(2)\mmod A(1)_* \otimes \ul{\tmf}_{j-1}  \to \Sigma^{8j+9} \ul{\bo}_{j-1} \to 0 \label{eq:boSES1},
\\
0 \to \Sigma^{8j} \ul{\bo}_j \otimes \ul{\bo}_1 \to \ul{\bo}_{2j+1}\to A(2)\mmod A(1)_* \otimes \ul{\tmf}_{j-1} \to 0 \label{eq:boSES2}
\end{gather}
Here, $\ul{\tmf}_j$ is the $j$th $\tmf$-Brown-Gitler comodule --- it is the subcomodule of 
$$ H_*(\tmf;\FF_2) \cong A\mmod A(2)_* = \FF_2[\xib_1^8, \xib_2^4, \xib_3^2, \xib_4, \ldots] $$
spanned by monomials of weight less than or equal to $8j$.

\begin{rmk}
Technically speaking, as is addressed in \cite[Sec.~7]{BHHM}, the comodules $A(2)\mmod A(1)_* \otimes \ul{\tmf}_{j-1}$ in the above exact sequences have to be given a slightly different $A(2)_*$-comodule structure from the standard one arising from the tensor product.  However, this different comodule structure ends up being $\Ext$-isomorphic to the standard one.  As we are only interested in Ext groups, the reader can safely ignore this subtlety.
\end{rmk}

The exact sequences (\ref{eq:boSES1}) and (\ref{eq:boSES2}) can be re-expressed as resolutions in the derived category of $A(2)_*$-comodules:
\begin{gather*}
\ul{\bo}_{2j}\to A(2)\mmod A(1)_* \otimes \ul{\tmf}_{j-1}  \to \Sigma^{8j+9} \ul{\bo}_{j-1} \to \Sigma^{8j} \ul{\bo}_j[2],
\\
\ul{\bo}_{2j+1}\to A(2)\mmod A(1)_* \otimes \ul{\tmf}_{j-1} \to \Sigma^{8j} \ul{\bo}_j \otimes \ul{\bo}_1[1] 
\end{gather*}
which give rise to spectral sequences
\begin{gather*}
E_1^{n,s,t} = 
\left\{
\begin{array}{ll}
\Ext^{s,t}_{A(1)_*}(\ul{\tmf}_{j-1} \otimes M), & n = 0, \\
\Ext^{s,t}_{A(2)_*}(\Sigma^{8j+9}\bou_{j-1} \otimes M[-1]), & n = 1, \\ 
\Ext^{s,t}_{A(2)_*}(\Sigma^{8j}\bou_j \otimes M), & n = 2,  \\
0, & n > 2 \\ 
\end{array}
\right\}
\Rightarrow \Ext^{s,t}_{A(2)_*}(\bou_{2j} \otimes M), \\
\\
E_1^{n,s,t} = \left\{
\begin{array}{ll}
\Ext^{s,t}_{A(1)_*}(\ul{\tmf}_{j-1} \otimes M), & n = 0, \\
\Ext^{s,t}_{A(2)_*}(\Sigma^{8j}\bou_j \otimes \bou_1 \otimes M), & n = 1, \\
0, & n > 1 \\ 
\end{array}
\right\}
\Rightarrow \Ext^{s,t}_{A(2)_*}(\bou_{2j+1} \otimes M). \\
\end{gather*}
These spectral sequences have been observed to collapse in low degrees (see \cite{BOSS}) but in general it seems possible they might not collapse.
They inductively build $\Ext_{A(2)_*}(\bou_i \otimes M)$ out of $\Ext_{A(2)_*}(\bou_1^{\otimes k} \otimes M)$ and $\Ext_{A(1)_*}(\ul{\tmf}_j \otimes M)$.

Define the $i$th $\bo$-Brown-Gitler polynomial $f_i(s,t,x)$ inductively by the formulas (inspired from the exact sequences by ignoring the $A(2)\mmod A(1)_*$ terms):
\begin{align*}
f_0 & = 1 \\
f_1 & = x \\
f_{2j} & = t^j f_j + st^{j+1}f_{j-1} \\
f_{2j+1} & = t^j x f_j 
\end{align*}
For a multi-index $I = (i_1, \ldots, i_n)$, define
$$ \bou_I := \bou_{i_1} \otimes \cdots \otimes \bou_{i_n} $$ 
and 
$$ f_I(s,t,x) = f_{i_1} \cdots f_{i_n}. $$
Write 
\begin{equation}\label{eq:f_I}
 f_{I} = \sum_{k,l,m} a(I)_{k,l,m} s^k t^l x^m. 
 \end{equation}
Then inductively the exact sequences give rise to a sequence of spectral sequences 
\begin{equation}\label{eq:boIss}
\left\{
\begin{array}{c}
\bigoplus_{k,l,m} \Ext^{s,t}_{A(2)_*}(\Sigma^{8l+k}\bou_1^{\otimes m}[-k] \otimes M)^{\oplus a(I)_{k,l,m}} \\
\oplus \Ext^{s,t}_{A(1)_*}(\text{something})
\end{array}
\right\}
\Rightarrow \Ext^{s,t}_{A(2)_*}(\bou_{I} \otimes M).
\end{equation}

The following facts about the polynomials $f_i(s,t,x)$ can be easily established by induction. 

\begin{lem} $\quad$
\begin{enumerate}
\item The coefficients in (\ref{eq:f_I}) satisfy $a(i)_{k,l,m} = 0$ unless $i = l+m$.
\item $f_i(s,t,x) \equiv t^{i-m} x^m \mod (s)$, where $m$ is the number of $1$'s in the dyadic expansion of $i$.
\item The highest power of $x$ appearing in $f_{i}(s,t,x)$ is less than or equal to the number of digits in the dyadic expansion of $i$.
\item The highest power of $s$ appearing in $f_i$ is the number of $1$'s to the left of the rightmost $0$ in the dyadic expansion of $i$. 
\end{enumerate}
\end{lem}

Finally, the following lemma explains that, in our $H(8,v_1^8)$ computations, we may disregard terms coming from $\Ext_{A(1)_*}$ in the sequence of spectral sequences (\ref{eq:boIss}).

\begin{lem}
In the algebraic $\tmf$-resolution for $M = H(8,v_1^8)$, the terms $$\Ext_{A(1)_*}(\text{something})$$ in (\ref{eq:boIss}) do not contribute to $\Ext^{s,t}_{A_*}(H(8,v_1^8))$ if
$$ s > \frac{1}{7}(t-s)+\frac{51}{7}. $$
\end{lem}

\begin{proof}
The connectivity of the $n$-line of the $\tmf$-resolution for $H(8,v_1^8)$ is $8n-1$, meaning that the bottom cell of the $n$-line contributes to $\Ext^{n,8n}_{A_*}$.  In particular, the contribution of the bottom cells rises on a line of slope $1/7$.  The groups $\Ext^{s,t}_{A(1)_*}(H(8,v_1^8))$ are displayed below.

\begin{center}
\includegraphics[height = .2\textheight]{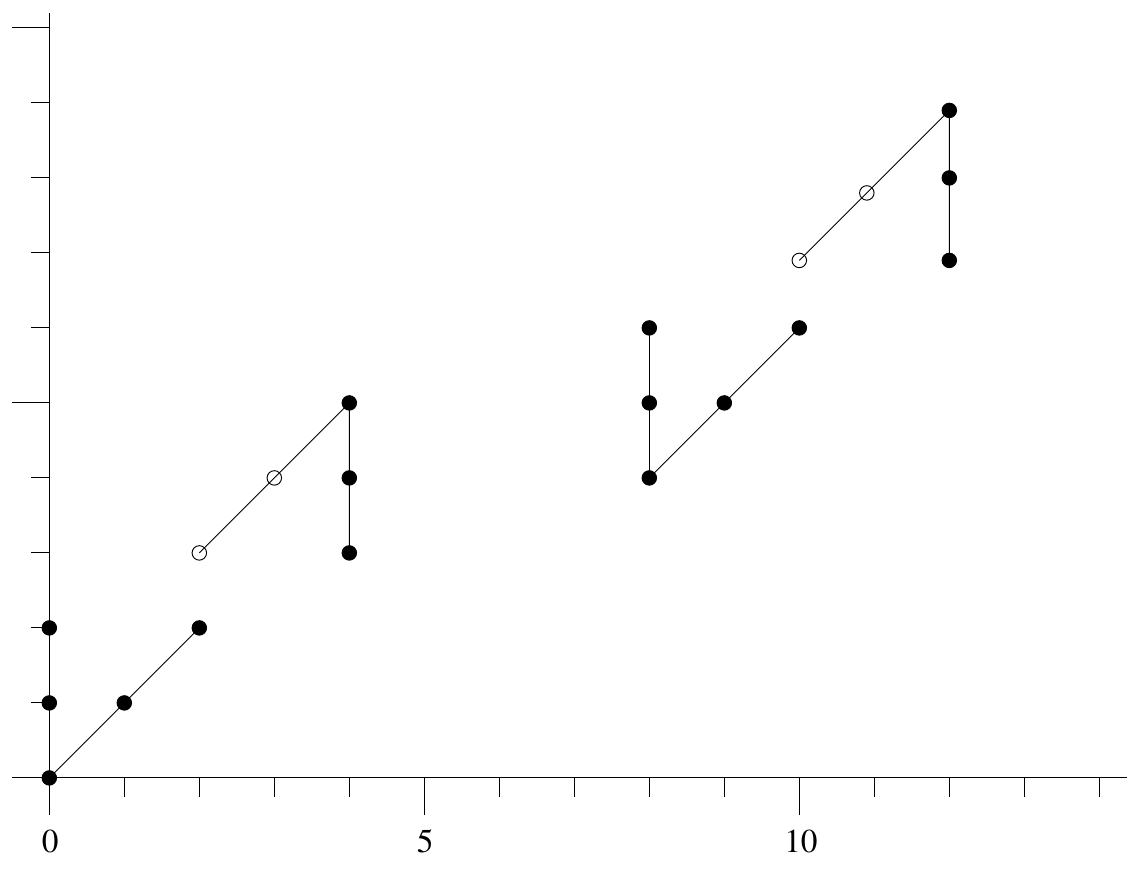}
\end{center}

The lowest line of slope $1/7$ which bounds the pattern above is 
$$ s = \frac{1}{7}(t-s)+\frac{51}{7}. $$
The result follows.
\end{proof}

\section{The MASS for $\tmf \wedge M(8,v_1^8)$}\label{sec:tmfM38}

\begin{figure}
\includegraphics[angle = 90, origin=c, height =.7\textheight]{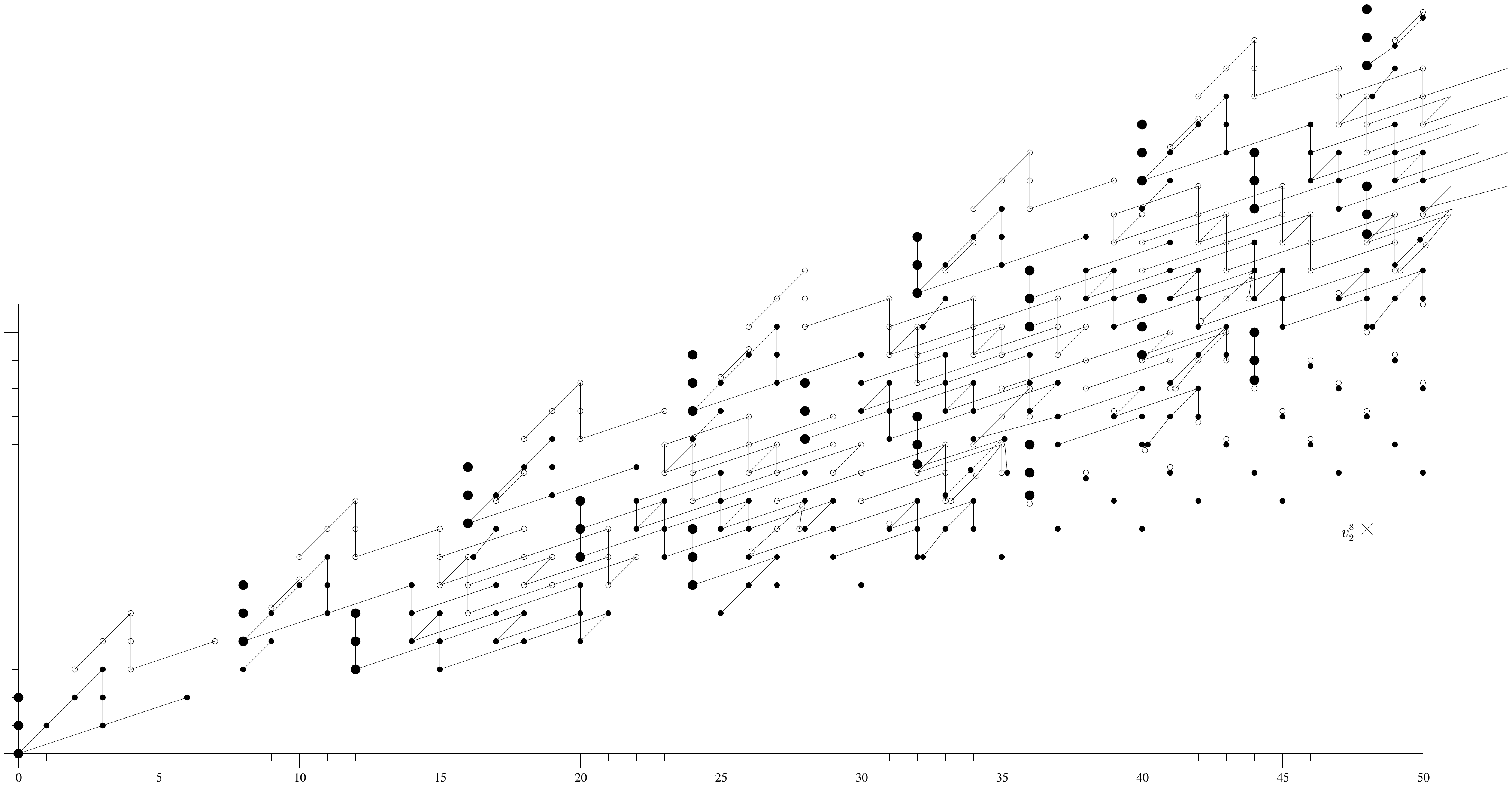}
\caption{The groups $\Ext_{A(2)_*}(H(8))$.}\label{fig:ExtH3}
\end{figure}

We now turn our attention to the analysis of the MASS for $\tmf \wedge M(8,v_1^8)$.  The groups $\Ext_{A(2)_*}(H(8))$ are easily computed using the computation of $\Ext_{A(2)_*}(\FF_2)$ (see, for example, the chart on p. 194 of \cite{tmf}) using the long exact sequence on $\Ext$ induced by the cofiber sequence
$$ \Sigma^{3}\FF_2[-3] \xrightarrow{h_0^3} \FF_2 \rightarrow H(8) \rightarrow \Sigma^3 \FF_2[-2]. $$
The result is displayed in Figure~\ref{fig:ExtH3}.  The computation is simplified by the fact that all $h_0$-torsion in $\Ext_{A(2)_*}(\FF_2)$ is $h_0^3$-torsion.  In this figure, solid dots correspond to classes carried by the ``$0$-cell'' of $H(8)$, and open circles correspond to classes carried by the ``$1$-cell'' of $H(8)$.  The large solid circles correspond to $h_0$-torsion free classes of $\Ext_{A(2)_*}(\FF_2)$ on the $0$-cell of $H(8)$.

\begin{figure}
\includegraphics[angle = 90, origin=c, height =.7\textheight]{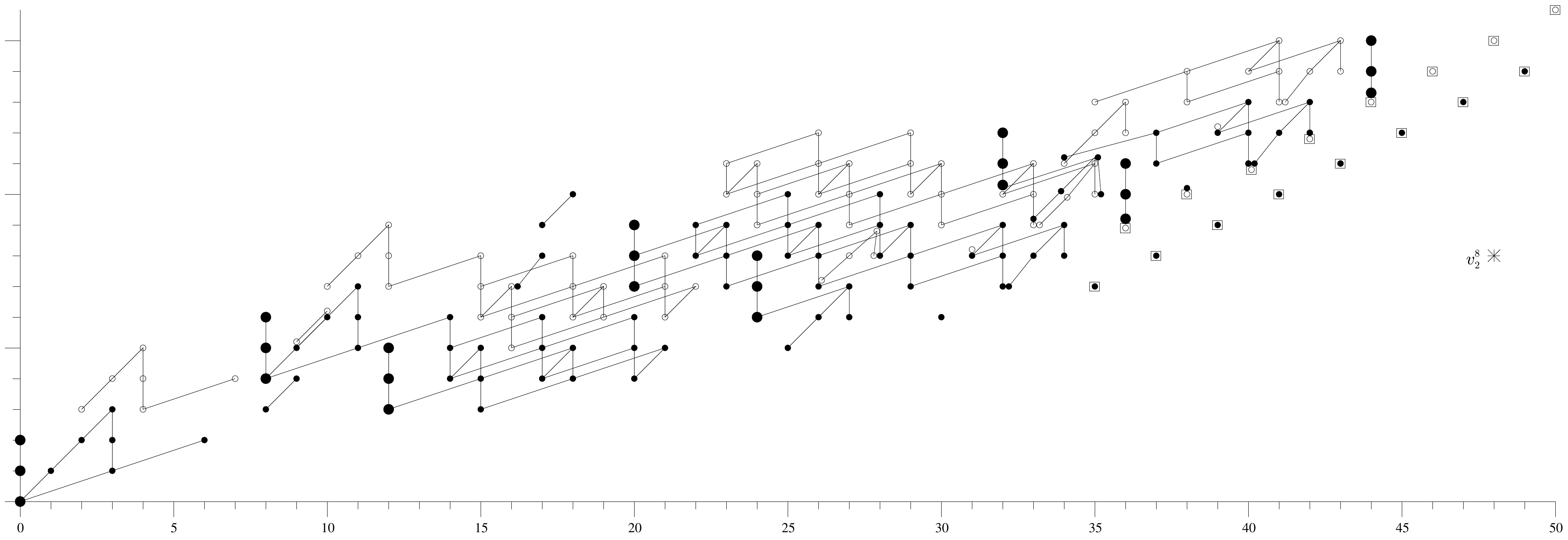}
\caption{The groups $\Ext_{A(2)_*}(H(8,v_1^8))$.}\label{fig:ExtH38}
\end{figure}

The computation of $\Ext_{A(2)_*}(H(8,v_1^8))$ is similarly accomplished by the long exact sequence on $\Ext$ given by the cofiber sequence
$$ \Sigma^{24}H(8)[-8] \xrightarrow{v_1^8} H(8) \rightarrow H(8,v_1^8) \rightarrow \Sigma^{24}H(8)[-7]. $$
Note that every class in $\Ext_{A(2)_*}(H(8))$ is $v_1^4$-periodic, so this computation is not difficult.  The result is depicted in Figure~\ref{fig:ExtH38}.  Here we retain the notation from Figure~\ref{fig:ExtH3} with regard to solid dots, large solid dots, and open circles.  The classes with solid boxes around them support $h_{2,1}$ towers, where $h_{2,1}$ corresponds to the class $[\xi^2_2]$ in the cobar complex for $A(2)_*$.  Everything is $v_2^8$-periodic.
  
\begin{figure}
\includegraphics[angle = 90, origin=c, height =.7\textheight]{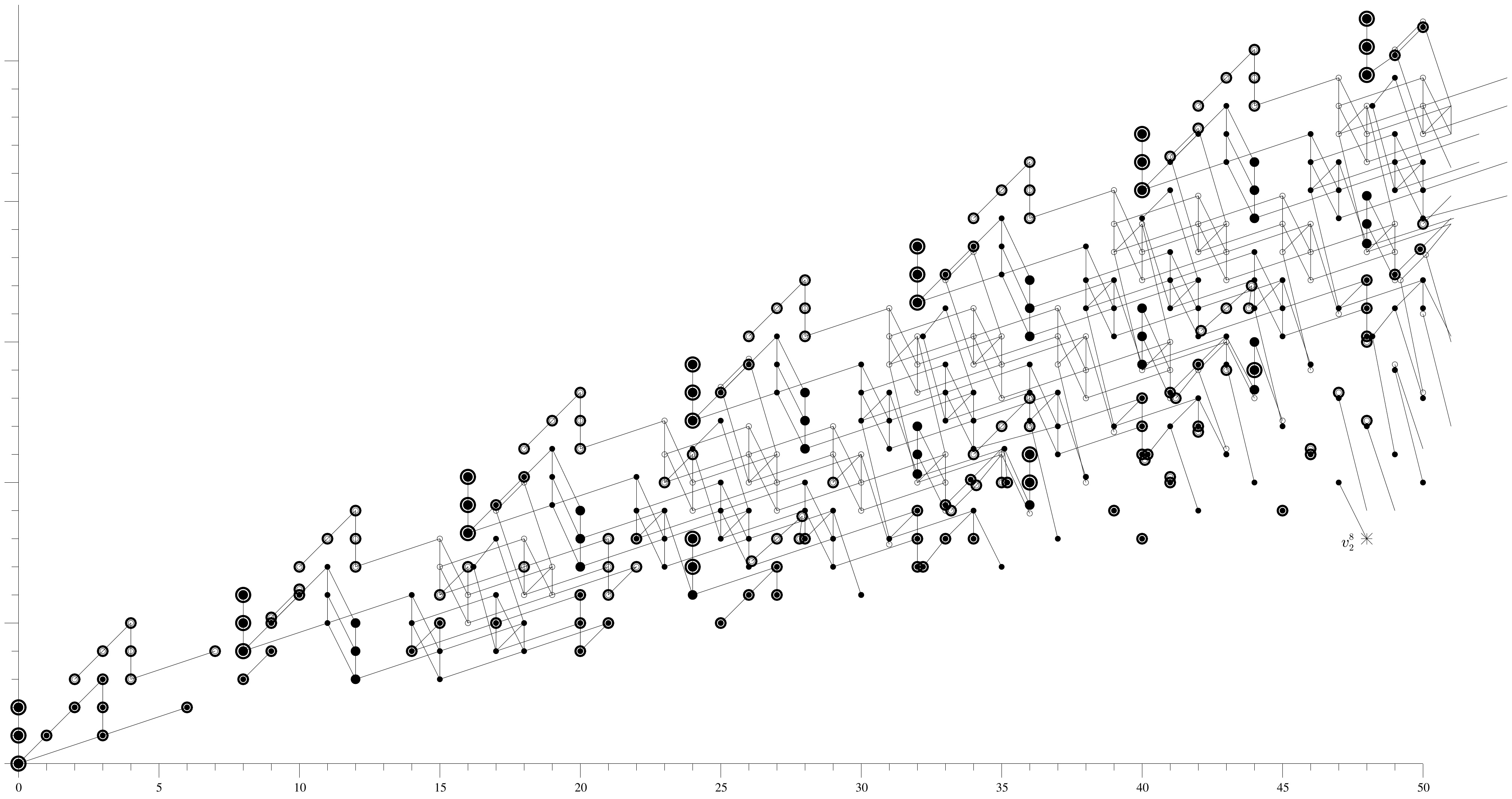}
\caption{The MASS for $\tmf \wedge M(8)$.}\label{fig:MASSM3}
\end{figure}

\begin{figure}
\includegraphics[angle = 90, origin=c, height =.7\textheight]{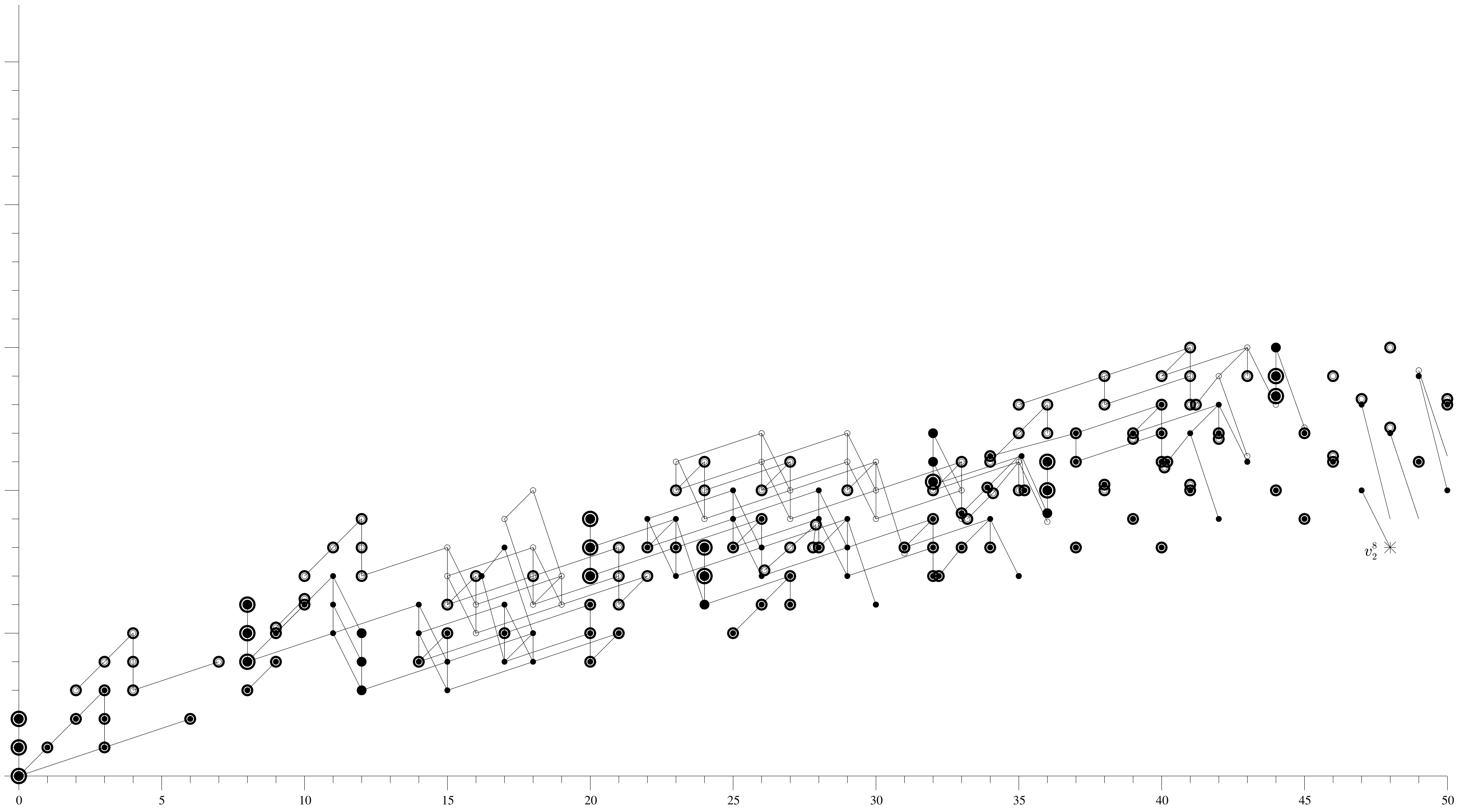}
\caption{The MASS for $\tmf \wedge M(8,v_1^8)$.}\label{fig:MASSM38}
\end{figure}

Figure~\ref{fig:MASSM3} depicts the differentials in the MASS for $\tmf \wedge M(8)$ through a range.  The complete computation of this MASS can be similarly accomplished, but it is not necessary for our purposes.  For the most part the differentials are deduced from the maps of spectral sequences induced by the maps
\begin{gather*}
\tmf \rightarrow \tmf \wedge M(8), \\
\tmf \wedge M(8) \rightarrow \Sigma \tmf.
\end{gather*}
The abutment of the spectral sequence, $\tmf_*M(8)$ is already easily computed from the long exact sequence associated to the cofiber sequence
$$ S^0 \xrightarrow{8} S^0 \rightarrow M(8), $$
and this information can be used to deduce the remaining differentials.
For the most part, these remaining differentials are related to hidden $\cdot 8$ extensions in the ASS for $\tmf$ via the geometric boundary theorem \cite[Lem.~A.4.1]{goodEHP}.  

\begin{ex}
The names we use for elements are those indicated in the chart on p.195 of \cite{tmf}.  We use $x[k]$, for $x$ an element of the ASS for $\tmf$, to denote ``$x$ on the $k$-cell''.
Via the geometric boundary theorem, the differential
$$ d_2((\ul{c_4\Delta+q})[1]) = \ul{8\Delta}\cdot \ul{(c_4+\epsilon)}[0] $$
arises from the hidden extension
$$ \ul{8\cdot (c_4\Delta+q)} =  \ul{8\Delta}\cdot \ul{(c_4+\epsilon)}. $$
\end{ex}

\begin{ex}
The most subtle differential in this range is
\begin{equation}\label{eq:diffdiff}
d_3(v_1^2 g^2[1]) = h_0^2\ul{2c_4c_6\Delta}[0].
\end{equation}
This differential does \emph{not} come from a hidden extension. In fact, in the ASS for $\tmf$ there is a differential
\begin{equation}\label{eq:dv12g2} 
d_4(v_1^2 g^2) = v_1^8 \cdot \ul{2\nu\Delta}.
\end{equation}
Naively, one might expect that differential~(\ref{eq:dv12g2}) lifts to a differential 
$$ d_4(v_1^2g^2[1]) =  v_1^8 \cdot \ul{2\nu\Delta}[1], $$
but that differential is preceded by (\ref{eq:diffdiff}).  Differential (\ref{eq:diffdiff}) is forced when we analyze the MASS for $\tmf \wedge M(8,v_1^8)$, and compare it to the answer predicted by the Atiyah-Hirzebruch spectral sequence for $\tmf_*M(8,v_1^8)$.
\end{ex}

Figure~\ref{fig:MASSM38} depicts the differentials in the MASS for $\tmf \wedge M(8,v_1^8)$ through the same range.  Again, the complete computation of this MASS can be similarly accomplished.  Remarkably, since the $E_2$-term of the MASS for $\tmf \wedge M(8)$ is entirely $c_4^2 = v_1^8$-periodic, there end up being no hidden $c_4^2$-extensions, and all the differentials in the MASS for $M(8,v_1^8)$ are simply given by the images of the differentials in the MASS for $M(8)$ under the map of spectral sequences
$$ \{ E_r^{*,*}(M(8))\} \rightarrow \{ E_r^{*,*}(M(8,v_1^8))\}. $$

\begin{figure}
\includegraphics[angle = 90, origin=c, height =.7\textheight]{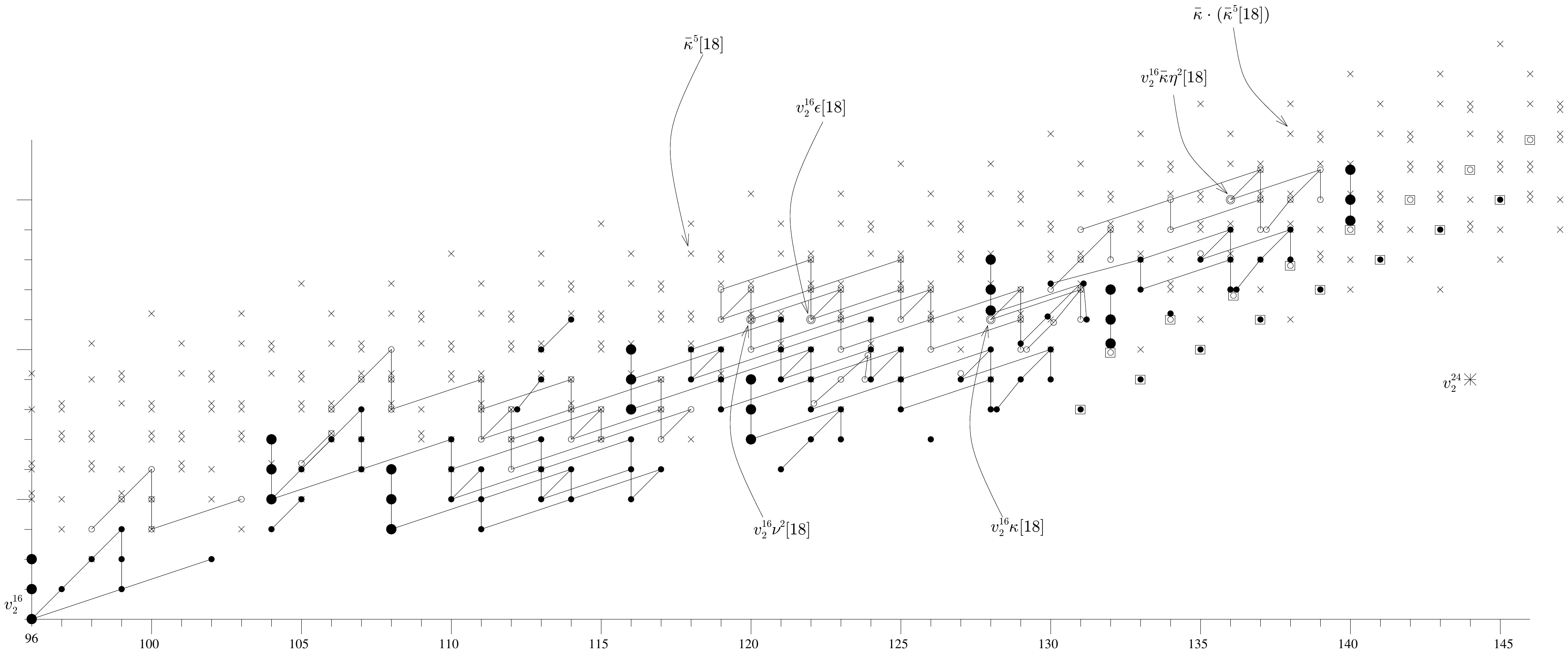}
\caption{$\Ext_{A(2)_*}(H(8,v_1^8))$ in the range 96-145, with classes which detect certain lifts of homotopy elements to the top cell of $M(8,v_1^8)$.}\label{fig:ExtH3896}
\end{figure}

Figure~\ref{fig:ExtH3896} displays $\Ext_{A(2)_*}(H(8,v_1^8))$ in the range 96-145.  In this figure, classes coming from $h_{2,1}$-towers starting in dimensions below 96 are labeled with $\times$'s.  We have circled the classes we are interested in, which detect lifts of certain homotopy elements to the top cell of $M(8,v_1^8)$.  These are determined using the differentials in the MASS for $M(8)$ and $M(8,v_1^8)$ using the geometric boundary theorem \cite[Lem.~A.4.1.]{goodEHP}.  We explain how this is done in one example.

\begin{ex}
Consider the class $v_2^{16}\epsilon = \Delta^4\epsilon$ in $\pi_{104}\tmf$.  It lifts to an element
$$ v_2^{16}\epsilon[1] \in \tmf_{105}M(8) $$
which is detected by the class
$$ v_2^{16}c_0[1] \in \Ext^{105+21,21}_{A(2)_*}(H(8)) $$
in the MASS for $\tmf \wedge M(8)$.
Now lift $v_2^{16}\epsilon[1]$ to an element
$$ v_2^{16}\epsilon[18] \in \tmf_{122}M(8,v_1^8). $$
We wish to determine an element of
$$ \Ext_{A(2)_*}(H(8,v_1^8)) $$
which detects one of these lifts in the MASS.  In the MASS for $\tmf \wedge M(8)$, there is a differential  
$$ d_3(\Delta^4 v_1^4 e_0[1]) = \Delta^4 v_1^8 c_0[1]. $$
It follows from the geometric boundary theorem that $v_2^{16}\epsilon[18]$ is detected by $\Delta^4 v_1^4 e_0[1]$ in the MASS for $\tmf \wedge M(8,v_1^8)$. 
\end{ex}

\section{The algebraic $\tmf$-resolution for $M(8,v_1^8)$}\label{sec:tmfres}

\begin{figure}
\includegraphics[angle = 90, origin=c, height =.7\textheight]{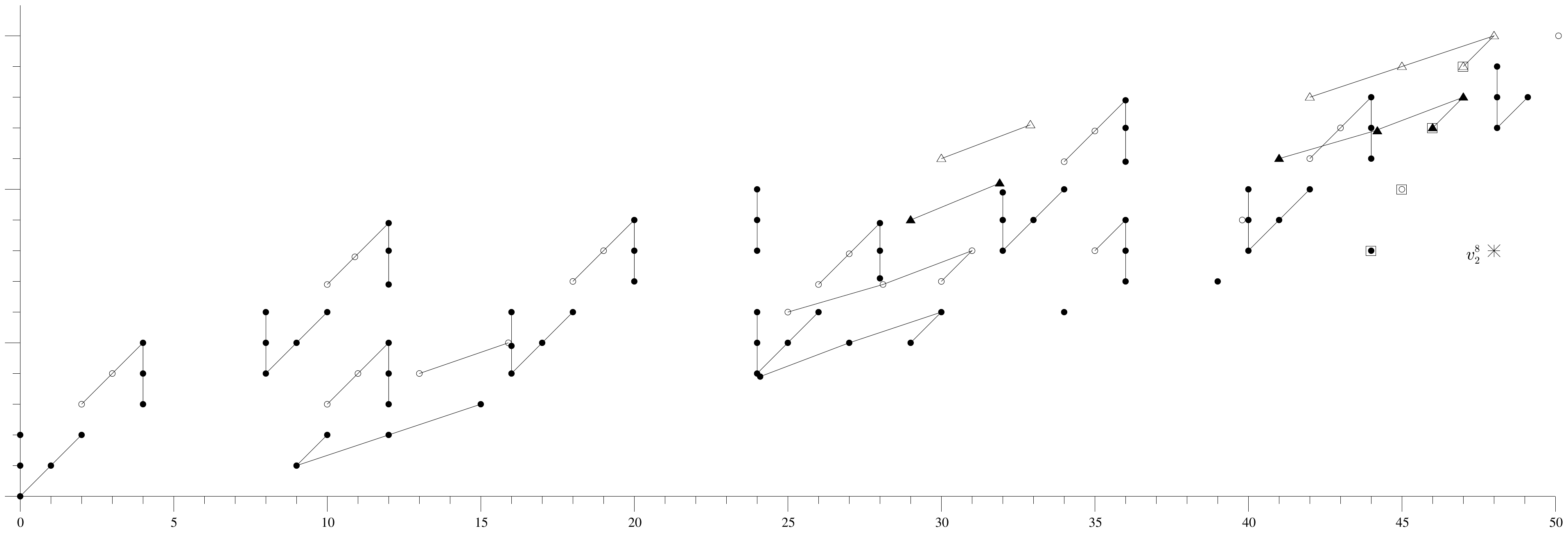}
\caption{$\Ext_{A(2)_*}(\bou_1 \otimes H(8,v_1^8))$.}\label{fig:Extbo1H38}
\end{figure}

For $n > 0$, and $i_1, \ldots, i_n > 0$, the terms
$$ \Ext_{A(2)_*}^{s,t}(\bou_{i_1} \otimes \cdots \otimes \bou_{i_n} \otimes H(8,v_1^{8})) $$
that comprise the terms in the algebraic $\tmf$-resolution for $H(8,v_1^8)$
are in some sense less complicated than $\Ext_{A(2)_*}(H(8,v_1^{8}))$.

Most of the features of these computations can already be seen in the computation of $\Ext_{A(2)_*}(\bou_1 \otimes H(8,v_1^8))$, which is displayed in Figure~\ref{fig:Extbo1H38}.  This computation was performed by taking the computation of $\Ext_{A(2)_*}(\bou_1)$ (see, for example, \cite{BHHM}) and running the long exact sequences in $\Ext$ associated to the cofiber sequences
\begin{gather*}
\Sigma^3 \bou_1 [-3] \xrightarrow{h_0^3} \bou_1 \rightarrow \bou_1 \otimes H(8), \\
\Sigma^{24}\bou_1 \otimes H(8)[-8] \xrightarrow{v_1^8} \bou_1 \otimes H(8) \rightarrow \bou_1 \otimes H(8,v_1^8). 
\end{gather*}
In Figure~\ref{fig:Extbo1H38}, as before, solid dots represent generators carried by the 0-cell of $H(8,v_1^8)$ and open circles are carried by the 1-cell.  Unlike the case of $\Ext_{A(2)_*}(H(8))$, there is $v_1^8$-torsion in $\Ext_{A(2)_*}(\bou_1 \otimes H(8))$.  This results in classes in $ \Ext_{A(2)_*}(\bou_1 \otimes H(8,v_1^8)) $ 
carried by the 17-cell and the 18-cell of $H(8,v_1^8)$, which are represented by solid triangles and open triangles, respectively.  A box around a generator indicates that that generator actually carries a copy of $\FF_2[h_{2,1}]$.  As before, everything is $v_2^8$-periodic.

One can similarly compute
$$ \Ext_{A(2)_*}(\bou_1^{\otimes k} \otimes H(8,v_1^8)) $$
for larger values of $k$ by applying the same method to the corresponding computations of 
$$ \Ext_{A(2)_*}(\bou_1^{\otimes k}) $$
in \cite{BHHM}.  We do not bother to record the complete results of these computations for small values of $k$, but will freely use them in what follows.  The sequences of spectral sequences $\ref{eq:boIss}$ imply these computations control $\Ext_{A(2)_*}(\bou_I)$.

In this section and the next, if
$$ \Ext_{A(2)_*}^{s,t}(\bou_{i_1} \otimes \cdots \otimes \bou_{i_n} \otimes H(8,v_1^{8})) $$
has a unique non-zero element, we shall refer to it as 
$$ x_{t-s,s}(i_1,\ldots,i_n). $$
We will now prove the following.

\begin{prop}\label{prop:tmfresPC}
The elements
$$ \ul{v_2^{16}\nu^2}[18], \:\ul{v_2^{16}\epsilon}[18],  \:\ul{v_2^{16}\kappa}[18], \:\ul{v_2^{16} \bar\kappa \eta^2}[18] \in \Ext^{s,t} _{A(2)}(H(8,v_1^8)) $$
of Figure~\ref{fig:ExtH3896} are permanent cycles in the algebraic $\tmf$-resolution for $H(8,v_1^8)$.
\end{prop}

\begin{cor}
The elements
$$ \ul{v_2^{16}\nu^2}[18], \ul{v_2^{16}\epsilon}[18],  \ul{v_2^{16}\kappa}[18], \ul{v_2^{16} \bar\kappa \eta^2}[18] \in \Ext^{s,t} _{A(2)}(H(8,v_1^8)) $$
lift to elements of $\Ext_{A_*}(H(8,v_1^8))$.
\end{cor}

\begin{proof}[Proof of Proposition~\ref{prop:tmfresPC}]
We begin by enumerating the targets of possible differentials supported by these classes in the algebraic $\tmf$-resolution
$$ \Ext_{A(2)_*}^{s,t}(\bou_{i_1} \otimes \cdots \otimes \bou_{i_n} \otimes H(8,v_1^{8})) \Rightarrow \Ext^{s+n, t+8(i_1+\cdots+i_n)}_{A_*}(H(8,v_1^8)). $$
Here, a $d_r$ differential raises $n$ by $r$.  Since the classes in question lie on the $n = 0$ line, the possible targets of differentials will all lie in terms with $n > 0$.
We begin with the ``edge'' case where $i_1 = i_2 = \cdots = i_n = 1$.
\begin{figure}
\includegraphics[angle = 90, origin=c, height =.7\textheight]{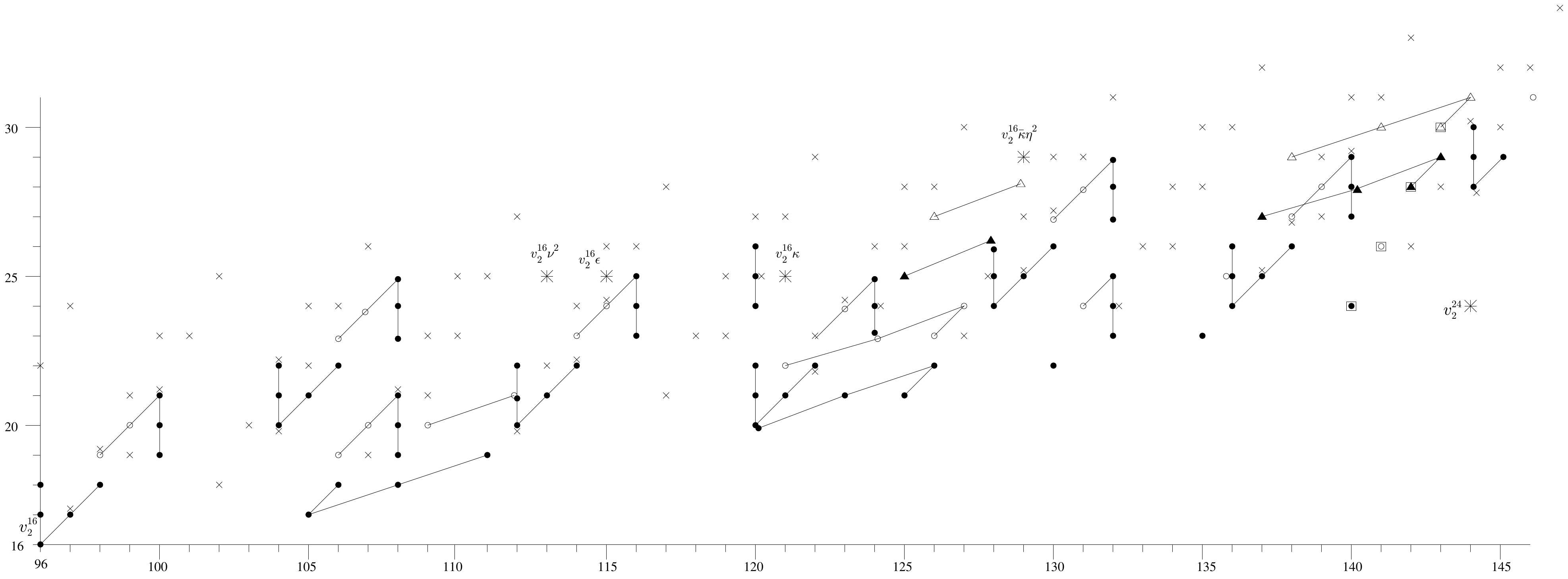}
\caption{$\Ext^{s,t}_{A(2)_*}(\bou_1 \otimes H(8,v_1^8))$ in the region $96 \le t-s \le 144$.}\label{fig:Extbo196H38}
\end{figure}

The first term to consider is $\Ext^{s,t}_{A(2)_*}(\bou_1 \otimes H(8,v_1^8))$, displayed in Figure~\ref{fig:Extbo196H38}.  These Ext groups can be easily determined from Figure~\ref{fig:Extbo1H38} by multiplying everything by $v_2^{16}$, and propagating the $h_{2,1}$-towers, and their $v_2^8$-multiples.  The classes belonging to $h_{2,1}$-towers are labeled with $\times$'s.  In this figure, we have indicated the relative position of the classes from the $n = 0$-term of the algebraic $\tmf$-resolution we wish to analyze with $\mathrlap{+}\times$'s.  We have set things up so that targets of differentials in the algebraic $\tmf$-resolution look like Adams $d_1$-differentials.  For example, we have  
$$ d_1(\ul{v_2^{16}\nu^2}[18]) \in \Ext^{26, 112+26}_{A(2)_*}(\bou_1 \otimes H(8,v_1^8))$$
and there are no non-trivial targets for such a differential.
Actually, the only possibility for a non-trivial $d_1$ on the classes in question hitting a class in $\Ext_{A(2)_*}(\bou_1\otimes H(8,v_1^8))$ is
$$ d_1(\ul{v_2^{16}\kappa}[18]) = x_{120,26}(1) \in \Ext^{26,120+26}_{A(2)_*}(\bou_1 \otimes H(8,v_1^8)). $$
However, since $\ul{\kappa} = d_0 \in \Ext_{A(2)_*}(\FF_2)$ lifts to a class
$$ d_0[18] \in \Ext_{A(2)_*}(H(8,v_1^8)), $$
using Lemma~\ref{lem:v2^8} we have
$$ d_1(\ul{v_2^{16}\kappa}[18]) = v_2^{16}d_1(d_0[18]). $$
Since $x_{120,26}(1)$ is not $v_2^{8}$ divisible, we deduce that it cannot be the target of such a differential.

\begin{figure}
\includegraphics[width =\textwidth]{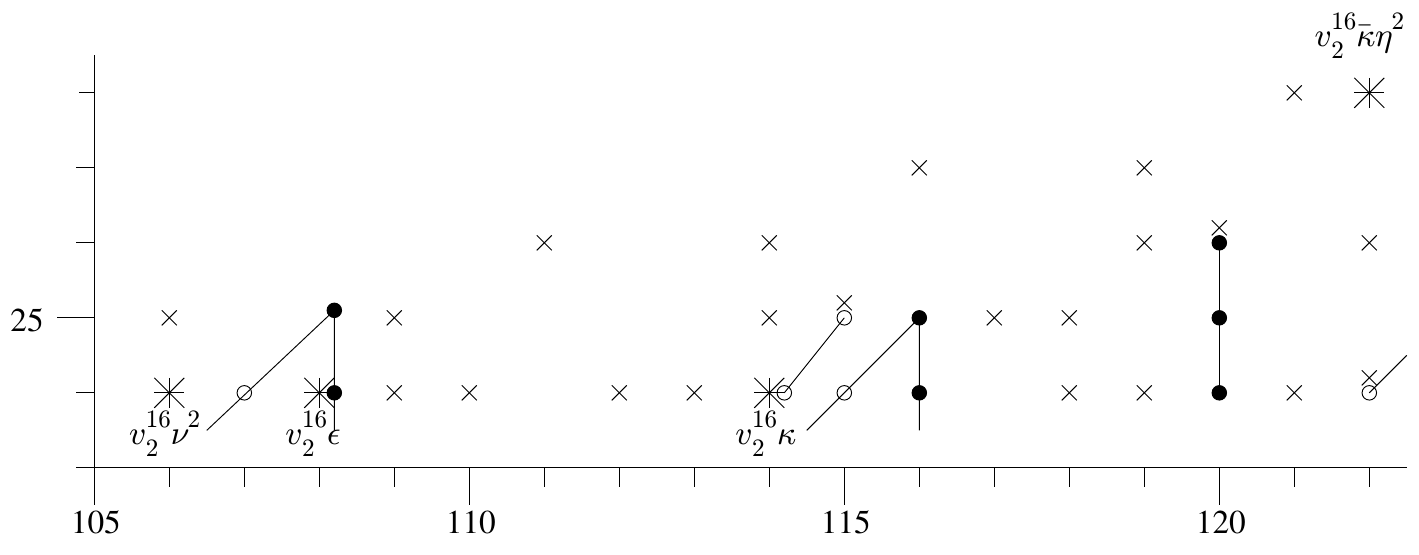}
\caption{Potential targets of differentials in $\Ext^{s,t}_{A(2)_*}(\bou_1^{\otimes 2} \otimes H(8,v_1^8))$}\label{fig:Extbo1296H38}
\end{figure}

\begin{figure}
\includegraphics[width =\textwidth]{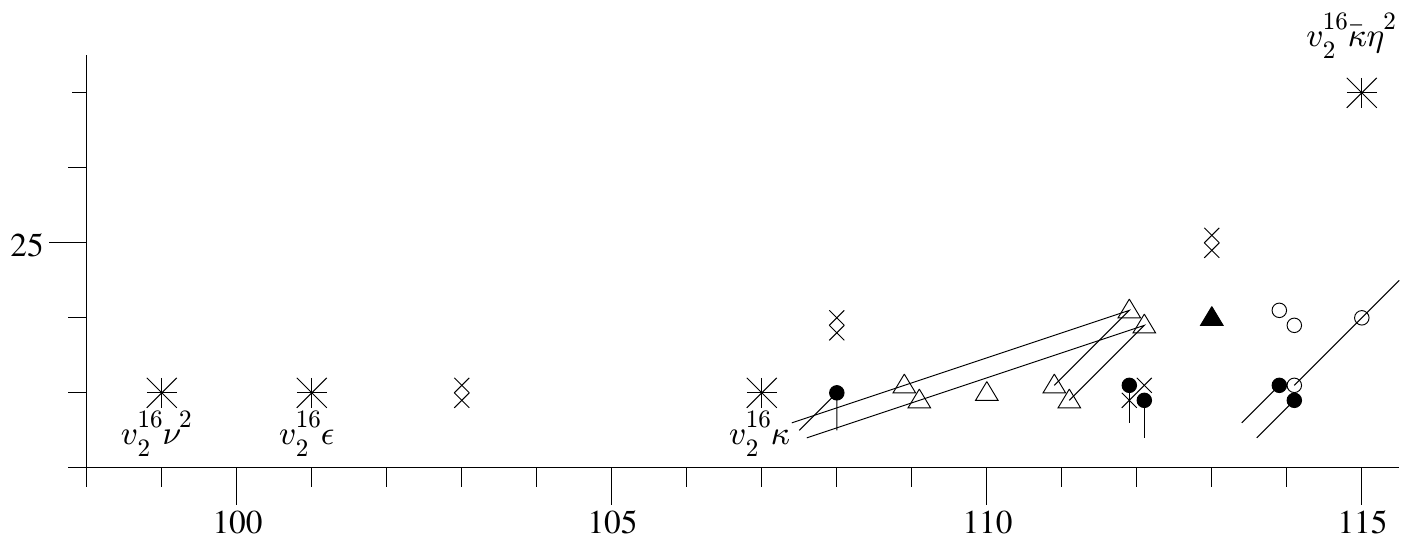}
\caption{Potential targets of differentials in $\Ext^{s,t}_{A(2)_*}(\bou_1^{\otimes 3} \otimes H(8,v_1^8))$}\label{fig:Extbo1396H38}
\end{figure}

Figures~\ref{fig:Extbo1296H38} and \ref{fig:Extbo1396H38} display the corresponding targets of potential $d_1$-differentials coming from the terms
$$ \Ext_{A(2)_*}(\bou_1^{\otimes 2} \otimes H(8,v_1^8)), \quad \Ext_{A(2)_*}(\bou_1^{\otimes 3} \otimes H(8,v_1^8)) $$
in the algebraic $\tmf$-resolution.  As the figures indicate, there are no possible non-zero targets of such $d_1$ differentials.  One finds that there are no contributions from
$$ \Ext_{A(2)_*}(\bou_1^{\otimes k} \otimes H(8,v_1^8)) $$
for $k \ge 4$ since the classes in question lie past the vanishing edge of these Ext groups.  

An elementary analysis, using the technology of Section~\ref{sec:boi}, shows that the classes lie beyond the vanishing edge of all of the other terms in the algebraic $\tmf$-resolution.
\end{proof}

\begin{rmk}
One could also attempt to prove Proposition~\ref{prop:tmfresPC} by showing the classes $\nu^2$, $\epsilon$, $\kappa$, and $\bar{\kappa}\eta^2$ lift to the top cell of $M(8,v_1^8)$, and then use Lemma~\ref{lem:v2^8}.  However, the obstruction that showed up in the case of $v_2^{16}\kappa$ also shows up in the case of $\kappa$, but is harder to eliminate because one no longer can appeal to $v_2^8$ divisibility.
\end{rmk}

\section{The MASS for $M(8,v_1^8)$}\label{sec:MASS}

We are now in a position to proof our main result.

\begin{thm}\label{thm:MASSPCs}
The elements
$$ \ul{v_2^{16}\nu^2}[18], \: \ul{v_2^{16}\epsilon}[18],  \: \ul{v_2^{16}\kappa}[18], \: \ul{v_2^{16} \bar\kappa \eta^2}[18] \in \Ext^{s,t} _{A_*}(H(8,v_1^8)) $$
of Figure~\ref{fig:ExtH3896} are permanent cycles in the MASS for $M(8,v_1^8)$.
\end{thm}

\begin{proof}
We analyze the potential targets of Adams differentials supported by these elements $x$ using the algebraic $\tmf$-resolution.  We can rule out any possible contributions from
$$ \Ext_{A(2)_*}(H(8,v_1^8)) $$
because if $d_r(x)$ is detected by $\Ext_{A(2)_*}(H(8,v_1^8))$, the map of MASS's induced by the map
$$ M(8,v_1^8) \rightarrow \tmf \wedge M(8,v_1^8) $$
would result in a corresponding non-trivial differential in the MASS for $\tmf \wedge M(8,v_1^8)$.  However, we already know that these classes $x$ exist in $\tmf_*M(8,v_1^8)$.  We are left with showing that $d_r(x)$ cannot be detected by
$$ \Ext_{A(2)_*}(\bou_{i_1} \otimes \cdots \otimes \bou_{i_n} \otimes H(8,v_1^8)) $$
in the algebraic $\tmf$-resolution.  Consulting the analysis of the proof of Proposition~\ref{prop:tmfresPC}, we see that the only possible targets of $d_r(x)$ could be detected in
$$ \Ext_{A(2)_*}(\bou_1 \otimes H(8,v_1^8)). $$
By Figure~\ref{fig:Extbo196H38}, these possibilities are:
\begin{align*}
d_2(\ul{v_2^{16}\nu^2}[18]) = x_{112,27}(1), \\
d_2(\ul{v_2^{16}\kappa}[18]) = x_{120,27}(1).
\end{align*}
In each of these two cases, the source is $g = h_{2,1}^4$-torsion, while the target is $g$-torsion free.
The source is checked to be $g$-torsion by checking there are no terms in higher filtration in the algebraic $\tmf$-resolution which could detect a $g$-tower supported by the source.  The target is checked to be $g$-torsion free, as the only possible way for a $g$-tower on the target to be truncated in the algebraic $\tmf$-resolution would be for it to be killed by a $g$-tower in $\Ext_{A(2)_*}(H(8,v_1^8))$.  The only possible $g$-towers which could do such a killing originate in stems smaller than the stems of the targets in question, and would therefore kill the targets themselves.
\end{proof}

\section{Tabulation of some non-trivial elements in Coker $J$}\label{sec:cokerJ}

We now use the $v_2$-periodic elements constructed in the previous section, as well as the existing literature on low dimensional computations of the stable stems, to show there exist non-trivial elements in $\coker J$ of Kervaire invariant $0$ in dimensions congruent to $0,-2,$ and $4$ modulo $8$ in most dimensions less than $140$.  We also summarize tentative results for dimensions $141-200$.

\subsection*{Results in dimensions less than 140}
 
Tables~\ref{tab1} (respectively~\ref{tab2}) list non-zero elements in $(\coker J_{4k})_{(p)}$ (respectively non-zero elements in $(\coker J_{8k-2})_{(p)}$ with Kervaire invariant $0$).
A $0$ entry indicates the group is known to be zero.  

\begin{table}
\begin{tabular}{c|c|c|c}
$4k$ & $p = 2$ & $p = 3$ & $p=5$ \\
\hline
4	&	0	&	0	&	0	\\
8	&	$\pmb{\epsilon}$	&	0	&	0	\\
12	&	0	&	0	&	0	\\
16	&	$\eta_4$	&	0	&	0	\\
20	&	$\pmb{\bar{\kappa}}$	&	$\beta_1^2$	&	0	\\
24	&	$\eta_4 \epsilon$	&	0	&	0	\\
28	&	$\pmb{\epsilon \bar{\kappa}}$	&	0	&	0	\\
32	&	$\pmb{\{q\}}$	&	0	&	0	\\
36	&	$\{ t \}$	&	$\beta_2 \beta_1$	&	0	\\
40	&	$\pmb{\bar{\kappa}^2}$	&	$\beta_1^4$	&	0	\\
44	&	$\{ g_2 \}$	&	0	&	0	\\
48	&	$\pmb{\kappa^2\bar{\kappa}}$  &	0	&	0	\\
52	&	$\pmb{\bar{\kappa} \{ q\}}$	&	$\beta_2^2$	&	0	\\
56	&	0	&	0	&	0	\\
60	&	$\pmb{\bar{\kappa}^3}$	&	0	&	0	\\
64	&	$\eta_6$	&	0	&	0	\\
68	&		&	$\langle \alpha_1, \beta_{3/2},\beta_2 \rangle$	&	0	\\
72	&		&	$\beta_2^2 \beta_1^2$	&	0	\\
76	&		&	0	&	$\beta_1^2$	\\
80	&	$\pmb{\bar{\kappa}^4}$	&	0	&	0	\\
84	&		&	$\beta_5 \beta_1$	&	0	\\
88	&	$\{ g_2^2 \}$	&	0	&	0	\\
92	&		&	$\beta_{6/3} \beta_1$	&	0	\\
96	&	$\eta_6 \{d_1\}$	&	0	&	0	\\
100	&	$\pmb{\bar{\kappa}^5}$	&	$\beta_2 \beta_5$	&	0	\\
104	&	$\pmb{v_2^{16} \epsilon}$	&		&	0	\\
108	&	$\eta_6 \{g_2\}$	&		&	0	\\
112	&		&	$\beta_{6/3} \beta_1^3$	&	0	\\
116	&	$\pmb{(v_2^{16} \kappa)\nu^2}$	&		&	0	\\
120	&	$\pmb{\bar\kappa^6}$	&		&	0	\\
124	&		&		&	$\beta_2 \beta_1$	\\
128	&	$\eta_7$	&		&	0	\\
132	&	$\{ h_2 h_6^2\} \nu$	&		&	0	\\
136	&	$\eta_7\epsilon$	&		&	0	\\
\end{tabular}
\caption{Some non-trivial elements of $(\coker J_{4k})_{(p)}$ in low degrees.}\label{tab1}
\end{table}

\begin{table}
\begin{tabular}{c|c|c|c}
$8k-2$ & $p = 2$ & $p = 3$ & $p=5$ \\
\hline
6	&	0	&	0	&	0	\\
14	&	$\pmb{\kappa}$	&	0	&	0	\\
22	&	$\pmb{\epsilon \kappa}$	&	0	&	0	\\
30	&	0	&	$\beta_1^3$	&	0	\\
38	&	$\{h_1 x\}$	&	$\beta_{3/2}$	&	$\beta_1$	\\
46	&	$\pmb{\{w\} \eta}$	&	$\beta_2 \beta_1^2$	&	0	\\
54	&	$\pmb{\bar{\kappa}^2\kappa}$	&	0	&	0	\\
62	&		&	$\beta_2^2 \beta_1$	&	0	\\
70	&	$\pmb{\langle \bar{\kappa} \{w\},\nu,\eta \rangle}$ &	0	&	0	\\
78	&		&	$\beta_2^3$	&	0	\\
86	&		&	$\beta_{6/2}$	&	$\beta_2$	\\
94	&		&	$\beta_5$	&	0	\\
102	&	$\pmb{v_2^{16} \nu^2}$	&	$\beta_{6/3} \beta_1^2$	&	0	\\
110	&	$\pmb{v_2^{16} \kappa}$	&		&	0	\\
118	&	$\pmb{(v_2^{16} \kappa)\epsilon}$	&		&	0	\\
126 &  $\bra{\theta_5, 2, \{X_2+C'\}}$ or $\theta_5\eta_6$ &		&	0	\\
134 & & & $\beta_3$ \\
\end{tabular}
\caption{Some non-trivial elements of $(\coker J_{8k-2})_{(p)}$ with Kervaire invariant $0$ in low degrees.}\label{tab2}
\end{table}

As the discussion in the last section indicates, a primary source of these non-zero elements are the $v_2$-periodic elements.  In fact, all of the $3$ and $5$-primary elements in Tables~\ref{tab1} and \ref{tab2} are $v_2$-periodic (in fact, $144$-periodic for $p = 3$ \cite{BehrensPemmaraju} and $48$-periodic for $p = 5$ \cite{Smith}).  For $p = 2$, the classes in boldface are known (or tentatively known) to be 192-periodic, and (with the exception of $\kappa^2 \bar\kappa$ and $\bar\kappa^6$) are detected by $\tmf$ via its Hurewicz homomorphism.

Thus, while we are emphasizing the low dimensional aspects of the subject, in fact the $v_2$-periodic classes are giving exotic spheres in infinitely many dimensions, of certain congruence classes mod $192$, $144$, and $48$.  All in all, in our range, over half of the candidates are coming from $v_2$-periodic classes.

In dimensions less than $60$ for $p = 2$, less than $104$ for $p = 3$, and in all dimensions depicted for $p = 5$, these non-trivial elements can be found in \cite{Isaksen}, \cite{Isaksen} and \cite{Ravenel}. 
The computation $(\pi_*^s)_{(3)}$ in the range up to 103 was first independently done by Nakamura \cite{Nakamura} and Tangora \cite{Tangora}.

Note that for any prime $p > 2$, the first non-trivial element of $(\coker J_*)_{(p)}$ is $\beta_1$ (see, for instance, \cite{Ravenel}), which lies in dimension $2(p^2-1)- 2(p-1)-2$.  The only prime greater than $5$ for which this dimension is less than $200$ is $p = 7$, for which $\abs{\beta_1} = 82$, and this is the only nontrivial $7$-torsion in $\coker J$ in our range.  This adds nothing to the discussion as $82 \equiv 2 \mod 8$.

The $2$-primary elements with names $v_2^{16}x$ for various $x$ were constructed in Section~\ref{sec:MASS}.  These classes are all non-trivial because they are detected in the $\tmf$-Hurewicz homomorphism.  In all fairness, in each of these dimensions except for dimension 118, we could have manually constructed non-trivial classes using certain products and Toda brackets involving $\theta_5$.  However, as we do not know another means of producing a class in dimension 118, we would still have to use the $\tmf$-resolution technique to handle that dimension.  We also have a preference for the classes we construct using the $\tmf$-resolution, as they are all $v_2$-periodic, and are either known or suspected to be $192$-periodic, Hence they actually account for an entire congruence class mod $192$ of dimensions.

Below, we explain the origin and non-triviality of the remaining $2$-primary elements in the tables.  Note that in our range the $E_2$-term of the $2$-primary ASS can be computed using software developed by Bruner \cite{Bruner}, \cite{BrunerExt}.

\begin{description}
\item[dim 60: $\bar\kappa^3$] This class is non-zero in $\pi_*\tmf$, hence non-zero.
\vspace{10pt}

\item[dim 64: $\eta_6$] The existence and non-triviality of this class was established by the fourth author \cite{etaj}.
\vspace{10pt}

\item[dim 70: $\langle \bar\kappa \{w\}, \nu, \eta \rangle$] The element $\{w \} \in \pi_{45}^s$ supports a hidden $\nu$-extension in the ASS for the sphere, so $\nu \{w\}$ is detected by $d_0e_0^2$ \cite[Lem.4.2.71]{Isaksen}.  Thus $\nu \bar\kappa \{w\}$ is detected in Adams filtration greater than or equal to $16$.  However, there are no non-trivial classes in $\pi^s_{68}$ with Adams filtration greater than or equal to $16$ \cite{Isaksenchart}.  This means the proposed Toda bracket exists.  The image of this Toda bracket in $\tmf$ does not contain zero, hence the Toda bracket in the sphere does not contain zero.
\vspace{10pt}

\item[dim 80: $\bar\kappa^4$] This class is non-zero in $\pi_*\tmf$, hence non-zero.
\vspace{10pt}

\item[dim 88: $\{g_2^2\}$] The element $g_2$ is a permanent cycle in the ASS.  Computer computations of $\Ext_{A_*}$ in this range reveal there are no possible sources of a differential to kill $g_2^2$. 
\vspace{10pt}

\item[dim 96: $\eta_6\{d_1\}$] The element $d_1$ is a permanent cycle in the ASS.  Its product with $\eta_6$ is detected in the ASS by the element $h_6h_1d_1$. There are no possible sources of a differential to kill this element.
\vspace{10pt}

\item[dim 100: $\bar\kappa^5$] This class is non-zero in $\pi_*\tmf$, hence non-zero.
\vspace{10pt}

\item[dim 108: $\eta_6\{g_2\}$] This element is detected in the ASS by the element $h_6h_1g_2$. There are no possible sources of a differential to kill this element.
\vspace{10pt}

\item[dim 120: $\bar\kappa^6$] See below.
\vspace{10pt}

\item[dim 126] See below.
\vspace{10pt}

\item[dim 128: $\eta_7$]  The existence and non-triviality of this class was established by the fourth author \cite{etaj}.
\vspace{10pt}

\item[dim 132: $\{h_2 h_6^2\}\nu$] Bruner showed that the classes $h_j^2 h_2$ are permanent cycles in the ASS.  There are no classes in $\Ext$ which can support a nontrivial Adams differential killing $h_2^2 h_6^2$.

\item[dim 136: $\eta_7\epsilon$] This class is detected by $\{c_0 h_7 h_1 \}$ in the ASS.  There are no classes in $\Ext$ which can support a non-trivial Adams differential killing this class.

\end{description}

\subsection*{The non-triviality of $\bar{\kappa}^6$}

Showing $\bar\kappa^6$ is non-zero is slightly tricky, since its image in $\pi_* \tmf$ is actually zero.  However, let $F(3)$ be the fiber
$$ F(3) \rightarrow \TMF \xrightarrow{q^*_3-f^*_3} \TMF_0(3) $$
for the maps $f_3^*$ and $q_3^*$ of Section~\ref{sec:tmf}.  Since $q_3^*$ and $f_3^*$ are $E_\infty$-ring maps, $F(3)$ is an $E_\infty$ ring spectrum, and in particular has a unit.

\begin{prop}
The image of $\bar\kappa^6$ under the map
$$ \pi_* S \rightarrow \pi_* F(3) $$
is non-zero.  In particular, $\bar\kappa^6$ is non-zero in $\pi_*^s$.
\end{prop}

\begin{proof}

We note that the element $\bar\kappa \in \pi^s_{20}$ factors over the 18-cell of the generalized Moore spectrum $M(8,v_1^8)$, hence so does $\bar{\kappa}^5$ --- we let 
$$ \bar\kappa^5[18] \in \pi_{118}M(8,v_1^8) $$
denote such a lift.  We let
\begin{align*}
\bar\kappa^5[18]_{F(3)} & \in F(3)_{118}M(8,v_1^8) \\
 \bar\kappa^5[18]_{\TMF} & \in \TMF_{118}M(8,v_1^8)
 \end{align*}
denote the images of this lift in $F(3)$ and $\TMF$-homology, respectively.
We can use the computations of $\tmf_*M(8,v_1^8)$ in Section~\ref{sec:tmfM38} to identify the term in the MASS which detects $\bar\kappa^5[18]_{\TMF}$, and we find that the product
\begin{equation}\label{eq:kappabar^6}
 \bar{\kappa} \cdot (\bar{\kappa}^5[18]_{\TMF}) \in \TMF_{138}M(8,v_1^8)
 \end{equation}
is non-trivial (see Figure~\ref{fig:ExtH3896}).  By manually computing the Atiyah-Hirzebruch spectral sequence for $\TMF_*(M(8,v_1^8))$, one finds that the element (\ref{eq:kappabar^6}) lifts to the element
\begin{equation}\label{eq:kappabar^6lift}
 c_4^2\Delta^5\eta[1]_{\TMF} \in \TMF_{138}M(8).
 \end{equation}
Applying the geometric boundary theorem \cite[Lem.~A.4.1, Case 5]{goodEHP} to the fiber sequence of resolutions of $F(3)$
$$
\scriptsize
\xymatrix@C-1em{
F(3) \wedge \Sigma^{16} M(8) \ar[r]^{v_1^8} \ar[d] &
F(3) \wedge M(8) \ar[r] \ar[d] &
F(3) \wedge M(8,v_1^8) \ar[r]^{\partial} \ar[d] &
F(3) \wedge \Sigma^{17} M(8)  \ar[d]  \\
\TMF \wedge \Sigma^{16} M(8) \ar[r]^{v_1^8} \ar[d]_{q-f} &
\TMF \wedge M(8) \ar[r] \ar[d]_{q-f} &
\TMF \wedge M(8,v_1^8) \ar[r] \ar[d]_{q-f} &
\TMF \wedge \Sigma^{17} M(8)  \ar[d]_{q-f}  \\
\TMF_0(3) \wedge \Sigma^{16} M(8) \ar[r]^{v_1^8}  &
\TMF_0(3) \wedge M(8) \ar[r]  &
\TMF_0(3) \wedge M(8,v_1^8) \ar[r]  &
\TMF_0(3) \wedge \Sigma^{17} M(8) 
}
$$
we see that $\partial(\bar\kappa \cdot \bar\kappa^5[18]_{F(3)})$ is computed by the usual formula for a connecting homomorphism: 
apply $q-f$ to the lift (\ref{eq:kappabar^6lift}), divide the result by $v_1^8$, and then take the image in $F(3)$.  By \cite[Prop.~8.1]{MahowaldRezk}, applying $q-f$ to (\ref{eq:kappabar^6lift}) gives
$$ (q-f)^*c_4^2\Delta^5\eta[1] = (a_3^{18}a_1^{14}\eta + \cdots)[1]. $$
Dividing by $v_1^8$, we get 
\begin{equation}\label{eq:kappabar^6F3}
 (a_3^{18}a_1^{6}\eta + \cdots)[1] \in \TMF_0(3)_{122}(\Sigma^{16}M(8)).
 \end{equation}
We deduce that $\partial(\bar\kappa \cdot \bar\kappa^5[18]_{F(3)})$ is detected by (\ref{eq:kappabar^6F3}).  Projecting to the top cell of $M(8)$, we deduce that the image of $\bar\kappa^6$ in $\pi_*F(3)$ is detected by
$$  a_3^{18}a_1^{6}\eta + \cdots \in \TMF_0(3)_{121}.$$
This class can be seen to be non-zero in $\pi_*F(3)$, by again appealing to \cite[Prop.~8.1]{MahowaldRezk}.
\end{proof}

\begin{rmk}
Since the map $S \rightarrow F(3)$ factors through the spectrum $Q(3)$ of \cite{MahowaldRezk}, \cite{BO}, the element $\bar\kappa^6$ is also detected in $Q(3)$. 
\end{rmk}

\subsection*{Dimension 126}

\begin{figure}
\includegraphics[angle = 0, origin=c, height =.5\textheight]{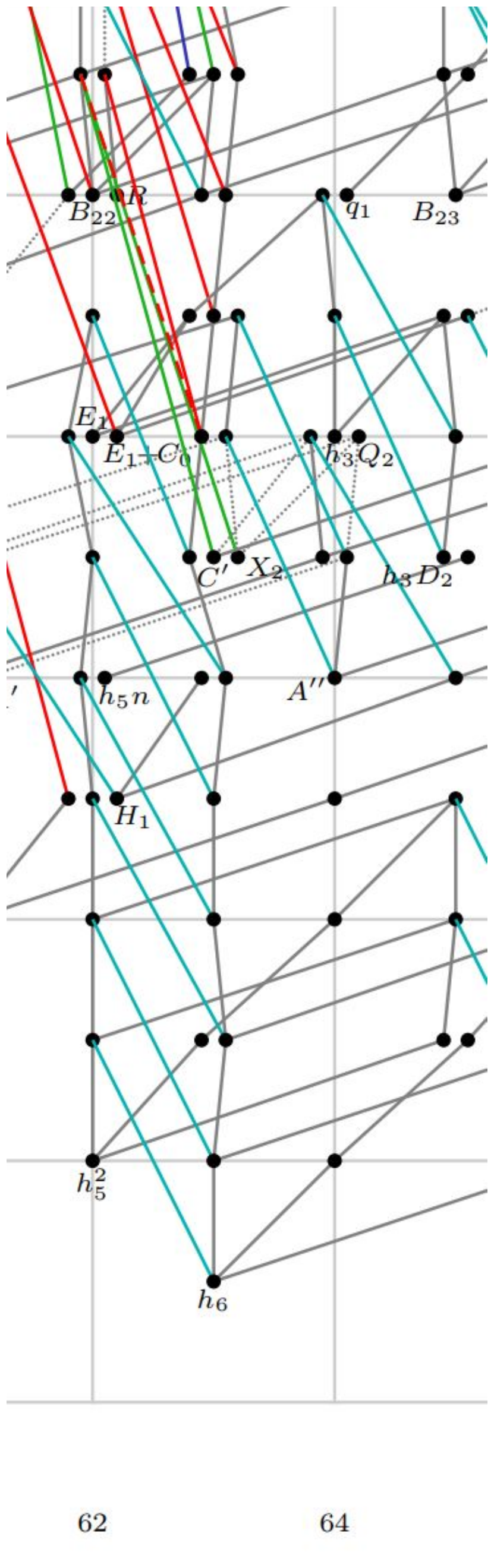}
\caption{The $2$-primary ASS for the sphere in the vicinity of $\theta_5$ (courtesy of Dan Isaksen).}\label{fig:ASS61}
\end{figure}  

Dimension 126 is handled by the following theorem, communicated to us by Isaksen and Xu.

\begin{thm}[Isaksen-Xu]\label{thm:126}
There exists a non-trivial element of $\coker J$ in dimension $126$ of Kervaire invariant $0$.
\end{thm}

\begin{proof}
This argument uses the structure of the Adams spectral sequence in the vicinity of $\theta_5$.  A chart displaying this region is depicted in Figure~\ref{fig:ASS61}.  Suppose that $\theta_5 \eta_6 \in \pi_{126}^s$ is non-trivial.  Then we are done.  Suppose however that $\theta_5 \eta_6 = 0$.
Consider the class $X_2 + C'$ in $\Ext^{7,63+7}_{A_*}(\FF_2, \FF_2)$.  A detailed analysis of the $2$-primary stable stems of Isaksen-Wang-Xu reveals that this class is a permanent cycle, and detects an element of order $2$ in $\pi_{63}^s$.  (We caution the reader that this analysis is not published at this time.)  Thus the Toda bracket
$$ \bra{\theta_5, 2, \{X_2 + C'\}} \in \pi^s_{126} $$
exists, and is detected by $h_6(X_2 + C')$ in the ASS.  We would be done if we could show that this class is not the target of an Adams differential.  There are only two classes which could support such a differential:
\begin{align*}
x_{6,99} & \in \Ext^{6,127+6}_{A_*}(\FF_2, \FF_2), \\
h_6^2 h_1 & \in \Ext^{3, 127+3}_{A_*}(\FF_2, \FF_2).
\end{align*}
Since $h_2 x_{6,99} = 0$ and $h_2 (X_2+C') \ne 0$ in $\Ext$, there cannot be a differential
$$ d_2(x_{6,99}) = X_2 + C'. $$
If our assumption on $\theta_5 \eta_6$ holds, the class $h_6^2 h_1$ would detect the Toda bracket
$$ \bra{2, \theta_5, \eta_6} $$
and hence be a permanent cycle.
\end{proof}

\subsection*{Tentative results in the range 141-200}

To edify the reader's curiosity, we conclude this paper with tables with some tentative results on exotic spheres beyond dimension 140.  The classes in dimensions 158, 160, 168, and 192 where pointed out by Zhouli Xu. A key resource are the computations of Christian Nassau \cite{Nassau}.

\subsection*{Some non-trivial elements of $(\coker J_{4k})_{(p)}$}$\quad$

\begin{center}
\begin{tabular}{c|c|c|c}
$4k$ & $p = 2$ & $p = 3$ & $p=5$ \\
\hline
140	&		&		&	0	\\
144	&	$\pmb{((v_2^{16} \eta w)\eta/2)\eta}$	&		&	0	\\
148	&	$\pmb{v_2^{16} \epsilon \bar\kappa}$	&		&	0	\\
152	&		&		&	$\beta_1^4$	\\
156	&	$\pmb{v_2^{16} 2\bar{\kappa}^3}$	&		&	0	\\
160	&	$\eta_5\eta_7$	&		&	0	\\
164	&	$\pmb{v_2^{24} \kappa \nu^2}$	&	$\beta_{10}\beta_1$	&	0	\\
168	&	$\eta_7\{f_1\}$	&		&	0	\\
172	&		&		&	$\beta_3 \beta_1$	\\
176	&		&		&	0	\\
180	&		&	$\beta_{11}\beta_1$	&	0	\\
184	&		&	$\beta_{10}\beta_1^3$	&	0	\\
188	&		&		&	0	\\
192	&	$\eta_6^3$	&		&	0	\\
196	&		&	$\beta_{11}\beta_2$	&	0	\\
200	&	$\pmb{v_2^{32}\epsilon}$	&		&	$\beta_{2}\beta_1^3$	\\
\end{tabular}
\end{center}

\subsection*{Some non-trivial elements of $(\coker J_{8k-2})_{(p)}$ with Kervaire invariant $0$}
\begin{center}
\begin{tabular}{c|c|c|c}
$8k-2$ & $p = 2$ & $p = 3$ & $p=5$ \\
\hline
142	&	$\pmb{v_2^{16} \eta w}$	&		&	0	\\
150	&	$\pmb{(v_2^{16} \epsilon \bar\kappa)\eta^2}$ 	&	&	0	\\
158	&	$\eta_7 \theta_4$ 	&		&	0	\\
166	&		&		&	0	\\
174	&	$\pmb{\beta_{32/8}}$	&	$\beta_{10}\beta_1^2$	&	0	\\
182	&	$\pmb{\beta_{32/4}}$	&	$\beta_{12/2}$	&	$\beta_4$	\\
190	&		&	$\beta_{11}\beta_1^2$	&	$\beta_1^5$	\\
198	&	$\pmb{v_2^{32} \nu^2}$	&		&	0	\\
\end{tabular}
\end{center}

The only dimensions in this range where we do not know if exotic spheres exist are 140, 166, 176, and 188.

\bibliographystyle{amsalpha}

\nocite{*}
\bibliography{exotic3}

\end{document}